\documentclass[10pt,leqno]{amsart}
\usepackage{amssymb}
\usepackage{xypic}
\usepackage{graphics}
\usepackage{array, graphicx}
\usepackage{tikz}
\usepackage{tikz-cd}
\usepackage [english]{babel}
\usepackage [autostyle, english = american]{csquotes}
\MakeOuterQuote{"}
\begin{document}
\theoremstyle{plain}
\newtheorem{thm}{Theorem}[section]
\newtheorem*{thm1}{Theorem 1}
\newtheorem*{thm2}{Theorem 2}
\newtheorem{lemma}[thm]{Lemma}
\newtheorem{lem}[thm]{Lemma}
\newtheorem{cor}[thm]{Corollary}
\newtheorem{prop}[thm]{Proposition}
\newtheorem{propose}[thm]{Proposition}
\newtheorem{variant}[thm]{Variant}
\newtheorem{conjecture}[thm]{Conjecture}
\theoremstyle{definition}
\newtheorem{notations}[thm]{Notations}
\newtheorem{notation}[thm]{Notation}
\newtheorem{rem}[thm]{Remark}  
\newtheorem{rmk}[thm]{Remark}
\newtheorem{rmks}[thm]{Remarks}
\newtheorem{defn}[thm]{Definition}
\newtheorem{ex}[thm]{Example}
\newtheorem{claim}[thm]{Claim}
\newtheorem{ass}[thm]{Assumption} 
\numberwithin{equation}{section}
\newcounter{elno}                \newcommand{\mc}{\mathcal}
\newcommand{\mb}{\mathbb}
\newcommand{\surj}{\twoheadrightarrow}
\newcommand{\inj}{\hookrightarrow} \newcommand{\zar}{{\rm zar}}
\newcommand{\an}{{\rm an}} \newcommand{\red}{{\rm red}}
\newcommand{\Rank}{{\rm rk}} \newcommand{\codim}{{\rm codim}}
\newcommand{\rank}{{\rm rank}} \newcommand{\Ker}{{\rm Ker \ }}
\newcommand{\Pic}{{\rm Pic}} \newcommand{\Div}{{\rm Div}}
\newcommand{\Hom}{{\rm Hom}} \newcommand{\im}{{\rm im}}
\newcommand{\Spec}{{\rm Spec \,}} \newcommand{\Sing}{{\rm Sing}}
\newcommand{\sing}{{\rm sing}} \newcommand{\reg}{{\rm reg}}
\newcommand{\Char}{{\rm char}} \newcommand{\Tr}{{\rm Tr}}
\newcommand{\Gal}{{\rm Gal}} \newcommand{\Min}{{\rm Min \ }}
\newcommand{\Max}{{\rm Max \ }} \newcommand{\Alb}{{\rm Alb}\,}
\newcommand{\GL}{{\rm GL}\,} 
\newcommand{\ie}{{\it i.e.\/},\ } \newcommand{\niso}{\not\cong}
\newcommand{\nin}{\not\in}
\newcommand{\soplus}[1]{\stackrel{#1}{\oplus}}
\newcommand{\by}[1]{\stackrel{#1}{\rightarrow}}
\renewcommand{\bar}{\overline}
\newcommand{\longby}[1]{\stackrel{#1}{\longrightarrow}}
\newcommand{\vlongby}[1]{\stackrel{#1}{\mbox{\large{$\longrightarrow$}}}}
\newcommand{\ldownarrow}{\mbox{\Large{\Large{$\downarrow$}}}}
\newcommand{\lsearrow}{\mbox{\Large{$\searrow$}}}
\renewcommand{\d}{\stackrel{\mbox{\scriptsize{$\bullet$}}}{}}
\newcommand{\dlog}{{\rm dlog}\,} 
\newcommand{\longto}{\longrightarrow}
\newcommand{\vlongto}{\mbox{{\Large{$\longto$}}}}
\newcommand{\limdir}[1]{{\displaystyle{\mathop{\rm lim}_{\buildrel\longrightarrow\over{#1}}}}\,}
\newcommand{\liminv}[1]{{\displaystyle{\mathop{\rm lim}_{\buildrel\longleftarrow\over{#1}}}}\,}
\newcommand{\norm}[1]{\mbox{$\parallel{#1}\parallel$}}
\newcommand{\into}{\hookrightarrow} \newcommand{\image}{{\rm image}\,}
\newcommand{\Lie}{{\rm Lie}\,} 
\newcommand{\CM}{\rm CM}
\newcommand{\sext}{\mbox{${\mathcal E}xt\,$}} 
\newcommand{\shom}{\mbox{${\mathcal H}om\,$}} 
\newcommand{\coker}{{\rm coker}\,} 
\newcommand{\sm}{{\rm sm}}
\newcommand{\tensor}{\otimes}
\renewcommand{\iff}{\mbox{ $\Longleftrightarrow$ }}
\newcommand{\supp}{{\rm supp}\,}
\newcommand{\ext}[1]{\stackrel{#1}{\wedge}}
\newcommand{\onto}{\mbox{$\,\>>>\hspace{-.5cm}\to\hspace{.15cm}$}}
\newcommand{\propsubset} {\mbox{$\textstyle{
\subseteq_{\kern-5pt\raise-1pt\hbox{\mbox{\tiny{$/$}}}}}$}}
\newcommand{\sA}{{\mathcal A}}
\newcommand{\sB}{{\mathcal B}} \newcommand{\sC}{{\mathcal C}}
\newcommand{\sD}{{\mathcal D}} \newcommand{\sE}{{\mathcal E}}
\newcommand{\sF}{{\mathcal F}} \newcommand{\sG}{{\mathcal G}}
\newcommand{\sH}{{\mathcal H}} \newcommand{\sI}{{\mathcal I}}
\newcommand{\sJ}{{\mathcal J}} \newcommand{\sK}{{\mathcal K}}
\newcommand{\sL}{{\mathcal L}} \newcommand{\sM}{{\mathcal M}}

\newcommand{\sN}{{\mathcal N}} \newcommand{\sO}{{\mathcal O}}
\newcommand{\sP}{{\mathcal P}} \newcommand{\sQ}{{\mathcal Q}}
\newcommand{\sR}{{\mathcal R}} \newcommand{\sS}{{\mathcal S}}
\newcommand{\sT}{{\mathcal T}} \newcommand{\sU}{{\mathcal U}}
\newcommand{\sV}{{\mathcal V}} \newcommand{\sW}{{\mathcal W}}
\newcommand{\sX}{{\mathcal X}} \newcommand{\sY}{{\mathcal Y}}
\newcommand{\sZ}{{\mathcal Z}} \newcommand{\ccL}{\sL}
 \newcommand{\A}{{\mathbb A}} \newcommand{\B}{{\mathbb
B}} \newcommand{\C}{{\mathbb C}} \newcommand{\D}{{\mathbb D}}
\newcommand{\E}{{\mathbb E}} \newcommand{\F}{{\mathbb F}}
\newcommand{\G}{{\mathbb G}} \newcommand{\HH}{{\mathbb H}}
\newcommand{\I}{{\mathbb I}} \newcommand{\J}{{\mathbb J}}
\newcommand{\M}{{\mathbb M}} \newcommand{\N}{{\mathbb N}}
\renewcommand{\P}{{\mathbb P}} \newcommand{\Q}{{\mathbb Q}}

\newcommand{\R}{{\mathbb R}} \newcommand{\T}{{\mathbb T}}
\newcommand{\Z}{{\mathbb Z}}

\title[Density function for the second coefficient of the Hilbert-Kunz 
function]
{Density function for the second coefficient of the Hilbert-Kunz 
function on Projective Toric Varieties}

\author{Mandira Mondal and Vijaylaxmi Trivedi }
\date{}
\address{Chennai Mathematical Institute\\ H1, SIPCOT IT Park, Siruseri\\
Kelambakkam 603103, India\\\\
School of Mathematics, Tata Institute of Fundamental Research\\
Homi Bhabha Road \\ Mumbai-400005, India}

\email{mandiram@cmi.ac.in; vija@math.tifr.res.in}

\thanks{}

\maketitle

\begin{abstract}
We prove that, analogous to the HK density function, (used for
studying  the Hilbert-Kunz multiplicity, 
the leading coefficient of the HK function),  there exists a 
$\beta$-density function $g_{R, {\bf m}}:[0,\infty)\longto \R$, where 
$(R, {\bf m})$ is the homogeneous coordinate ring associated to the toric pair
$(X, D)$, such that 
$$\int_0^{\infty}g_{R, {\bf m}}(x)dx = \beta(R, {\bf m}),$$
where $\beta(R, {\bf m})$ is the second coefficient of the Hilbert-Kunz 
function for $(R, {\bf m})$, 
as constructed by Huneke-McDermott-Monsky

Moreover we prove, (1)~~the
function $g_{R, {\bf m}}:[0, \infty)\longto \R$ is compactly supported
and is  continuous
 except at finitely many points, (2)~~the function $g_{R, {\bf m}}$ is multiplicative 
for the Segre products with the expression
involving the first two coefficients of the Hilbert polynomials of
the rings involved.

Here we also prove  and use a result (which is a refined 
version of a result by Henk-Linke) on the boundedness of the coefficients
 of rational Ehrhart quasi-polynomials of  convex rational  polytopes.
\end{abstract}
\section{Introduction}
Let $R$ be a Noetherian ring of dimension $d$ and prime characteristic 
$p$, and let $I\subset R$\ be an ideal such that $\ell{(R/I)}<\infty$.
Let $M$ be a finitely generated $R$-module. Then the Hilbert-Kunz function 
of $M$ with respect to $I$ is defined by 
$$\mbox{HK}(M, I)(n)=\ell(M/I^{[q]}M)$$
where $q=p^n$, the ideal $I^{[q]}=n$-th Frobenius power of $I$, and 
$\ell(M/I^{[q]}M)$ denotes the length of the $R$-module $M/I^{[q]}M$. 
The limit $$\text{lim}_{n\to\infty}\frac{\ell(M/I^{[p^n]}M)}{q^{d}}$$
exists (\cite{M}~1983) and is called the Hilbert-Kunz multiplicity of $M$ 
with respect to the ideal $I$ (denoted by $e_{HK}(M, I)$). Thus 
$HK(M,I)(n) = e_{HK}(M,I)q^d+O(q^{d-1})$. The 
Hilbert-Kunz multiplicity has been studied by many people since then. 

In 2004, Huneke-McDermott-Monsky (\cite{HMM2004}) proved the 
existence of a 2nd coefficient for the Hilbert Kunz function:
\vspace{.2 cm}

\begin{thm}\label{HMM} (Theorem~1 of \cite{HMM2004}):
Let $R$ be an excellent normal Noetherian ring of dimension $d$ 
and characteristic $p$ and let $I\subset R$ be an ideal such that 
$\ell{(R/I)}<\infty$. Let $M$ be a finitely generated $R$-module. 
Then there exists a real number $\beta(M, I)$ such that
$$\mbox{HK}(M, I)(n)=e_{HK}(M, I)q^{d}+ \beta(M, I)q^{d-1}+O(q^{d-2}).$$
\end{thm}

 They also found a relation with the divisor class group. This invariant 
was further studied by Kurano in [Ku] and 
he proved there that $\beta(M, I) =0$ if $R$ is $\Q$-Gorenstein ring and $M$ is a 
Noetherian $R$-module of finite projective dimension.  
Later the above theorem of Huneke-McDermott-Monsky was generalised 
from normal rings to the rings satisfying $(R_1)$ condition by Chan-Kurano in [CK] (also independently by 
Hochster-Yao in [HY]). 
Later Bruns-Gubeladze in [BG] have proved that 
HK function is a quasi polynomial and gave another proof of the 
existence of the constant second coefficient 
$\beta(R, {\bf m})$ for a  normal affine monoid, 

In order to study $e_{HK}(R, I)$,  
when  $R$ is a standard graded ring ($\dim~R\geq 2)$ 
and $I$ is a homogeneous ideal of finite colength, 
the second author (in \cite{Tri2015}) has defined the notion of 
Hilbert-Kunz Density function, and obtained its relation with the 
\mbox{HK} multiplicity 
(stated in this paper as Theorem~\ref{hkd}): 
{\it The \textnormal{\mbox{HK}} density function is a compactly supported continuous function 
$f_{R,I}:[0,\infty) \longto 
\R_{\geq 0}$ such that 
$$e_{HK}(R, I)=\int_0^\infty f_{R, I}(x)dx.$$

Moreover there exists a sequence of functions $\{f_n(R,I):[0,\infty)\longto 
\R_{\geq 0}\}_n$
given by 
$$f_n(R, I)(x) = \frac{1}{q^{d-1}}\ell(R/I^{[q]})_{\lfloor xq\rfloor}$$
such that $f_n$ converges uniformly to $f_{R,I}$.}

The existence of a  uniformly converging sequence  makes 
 the density function a more refined invariant (compared to $e_{HK}$) 
in the graded situation, 
and  a useful tool, {\it e.g.}, in  suggesting 
a simpler approach to the HK multiplicity  in characteristic $0$ (see [T4]), in
studying the asymptotic growth of 
$e_{HK}(R, {\bf m}^k)$ as $k\to \infty$ (see [T3]).
Applying  the theory of HK density functions  to  projective toric varieties 
(denoted here as toric pairs $(X,D)$), one obtains 
(Theorem~6.3 of [MT]) an algebraic  characterization of the tiling propery of 
the associated polytopes $P_D$ (in the ambient lattice) 
in terms of such  asymptotic behaviour of $e_{HK}$.

In the light of Theorem \ref{hkd}, one can speculate  whether 
there exists a similar `$\beta$-density function' 
$g_{R, {\bf m}}: [0,\infty)\to \R$  such that 
$$\int_{0}^{\infty}g_{R, {\bf m}}(x)\ dx= \beta(R, I),$$
which may similarly refine the $\beta$-invariant of [HMM] in the  graded 
case.

We find that this is indeed true for a toric pair $(X, D)$, {\it i.e.},
$X$  a projective toric variety  
over an algebraically closed field $K$ of characteristic $p>0$, with a
very ample $T$-Cartier divisor $D$. Let $R$ be the homogeneous coordinate ring of
$X$, with respect to the embedding given by the very ample
line bundle $\mathcal{O}_X(D)$,
and let $\textbf{m}$ be the homogeneous maximal ideal of $R$. We also use the 
notation $\mbox{rVol}_{d}$ to denote the ${d}$-dimensional 
relative volume function (see Definition~\ref{rv}).

We construct such a $\beta$-density function $g_{R, {\bf m}}$ as 
a limit of a `uniformly' 
converging sequence  of functions
$\{g_n:[0,\infty)\longto \R\}_{n\in \N}$, which are given by 
\begin{equation}\label{*g}
g_n(\lambda)=\frac{1}{q^{d-2}}\left(
\ell(R/\textbf{m}^{[q]})_{\lfloor \lambda q\rfloor}
-{\tilde f_n}(\lambda) q^{d-1}\right),\end{equation}
where ${\tilde f_n}(\lambda) = f_{R, {\bf m}}({\lfloor \lambda q\rfloor}/q)$
and $f_{R,{\bf m}}$ denotes the \mbox{HK} density function 
for $(R,{\bf m})$. From the construction it follows 
that $g_n$ is a compactly supported function.

Recall that given a toric pair $(X, D)$, of dimension $d-1$,  
there is a convex 
lattice polytope 
$P_D$, a convex polyhedral cone $C_D$  and a bounded body 
$\sP_D$ as in Notations~\ref{n1} (such a bounded body was introduced by 
K. Eto (see [Et]), in order to study the HK multiplicity for a toric ring, and  
he proved there that $e_{HK}$ is the relative volume of such a body). 
In [MT], it was  shown that the HK density function at $\lambda $ is 
the relative volume of the $\{z=\lambda\}$ slice of $\sP_D$. Here we prove that 
$\beta$-density function at $\lambda$ is expressible in terms of the 
relative volume of the $\{z=\lambda\}$ slice of the boundary, $\partial(\sP_D)$, 
of $\sP_D$.
 
In this paper  the following  is the main result. 

\vspace{5pt}

\noindent{\bf Main Theorem}.~
{\it  Let $(R, {\bf m})$ be the  
homogeneous coordinate ring of dimension $d\geq 3$, associated 
to the toric pair $(X, D)$.
Then  there exists a finite set $v(\sP_D)\subseteq 
\R_{\geq 0}$ such that, for any compact set 
$V\subseteq \R_{\geq 0}\setminus v(\sP_D)$, the sequence $\{g_n|_V\}_n$ 
converges uniformly  to $g_{R, {\bf m}}|_V$, where 
$g_{R,{\bf m}}:\R_{\geq 0}\setminus v(\sP_D)\longto \R$ is a continuous 
function  given by 
$$g_{R, {\bf m}}(\lambda) = 
 \mbox{rVol}_{d-2}\left(\partial({\sP_D})\cap \partial(C_D)
\cap \{z=\lambda\}\right)
- \frac{\mbox{rVol}_{d-2}\left(\partial({\sP_D})\cap 
\{z=\lambda\}\right)}{2}.$$

Moreover, for $q =p^n$, we have 
 $$\int_0^\infty g_n(\lambda)d\lambda = \int_0^\infty 
g_{R, {\bf m}}(\lambda)d\lambda +O(\frac{1}{q})~~\mbox{and}~~ 
\int_0^\infty {\tilde f_n}(\lambda)d\lambda = \int_0^\infty 
f_{R, {\bf m}}(\lambda)d\lambda +O(\frac{1}{q^2}).$$}

As a consequnce we get the following 
\begin{cor}\label{thebeta}{\it With the notations as above, we have 
$$\beta(R, {\bf m}) = \int_0^\infty g_{R, {\bf m}}(\lambda)d\lambda = 
 \mbox{rVol}_{d-1}\left(\partial({\sP_D})\cap \partial(C_D)
\right)
- \frac{\mbox{rVol}_{d-1}\left(\partial({\sP_D})\right)}{2}$$
and  the Hilbert-Kunz function of $R$ with 
respect to the maximal ideal ${\bf m}$ is  given by
 $$HK(R, {\bf m})(q) = e_{HK}(R, {\bf m})q^d +
\beta(R, {\bf m})q^{d-1}+O(q^{d-2}).$$}
\end{cor}

\vspace{5pt}
Note that we can write 
$$g_n(\lambda) = {\#(q{\mathcal P}_D\cap\{z= {\lfloor 
\lambda q\rfloor}\})}/q^{d-2}
-(q)f_{R,{\bf m}}({\lfloor \lambda q\rfloor}/q),$$
where $\#$ denotes the number of lattice points.

We show in Section~3 that 
  ${\bar{\mathcal P}_D} = (\cup_jP_j) \setminus 
(\cup_{j, \nu}E_{j_\nu})$, where $P_j$ and $E_{j_\nu}$ are  certain 
rational convex polytopes with proper intersections.
Then by applying the theory of Ehrhart quasi-polynomials and 
exhibiting that  (in the case of a toric pair)
the second coefficients of relevant  Ehrhart quasi-polynomials 
are constant, we deduce 
that for $x \in S=\{m/p^n\mid m,n\in \N\}\setminus v(\sP_D)$,
the sequence $\{g_n(x)\}_n$ is  convergent and converges pointwise to 
$g_{R, {\bf m}}$.
However we still know neither the existence of
$\lim_{n\to \infty}g_n(x)$, for every $x \in [0, \infty)$ (or for all 
$x$ except at finite number of points), nor
that this limit is a continuous function  on the dense set $S$. On
the other hand, for $\lambda_n := {\lfloor\lambda q\rfloor}/q \in S$, we have 
$$g_n(\lambda) = g_{R,{\bf m}}(\lambda_n)
+{\tilde c}(\lambda_n)/q,$$
where ${\tilde c}(\lambda_n)$ involves coefficients of  Ehrhart
quasi-polynomials of facets of $P_j\cap \{z=\lambda_n\}$.
Therefore, to achieve a `uniform convergence', we needed to prove 
the following: 

\vspace{5pt}

\noindent{\bf Theorem}~\ref{tcb}\quad {\it For a rational  convex 
polytope $P$, where $P_\lambda := P\cap \{z = \lambda\}$ and its Ehrhart
quasi-polynomial is given by 
$$i(P_\lambda, q) =
\sum_{i=0}^{\dim(P)}C_i(P_\lambda, q) q^i, \quad\mbox{if}~~\lambda q\in 
\Z_{\geq 0},$$
there exist constant ${\tilde c_i}(P)$ such that
every $|C_i(P_\lambda, q)|\leq {\tilde c_i}(P)$, for all 
 $q\lambda \in \Z_{\geq 0}$ and $q = p^n$}.

We prove the result using the theory of lattice points
in non-negative rational Minkowski sums. 
In fact we prove a general result 
 about convex rational polytopes: Recall that, for rational convex 
polytopes 
$P_1, P_2\subset \R^d$ with $\dim~(P_1+P_2) = d$, a 
well known result of McMullen implies that the function  
$Q(P_1,P_2;-):\Q^2_{\geq 0} \longto \Z$, given by 
$Q(P_1, P_2,;{\bf r}) = \#((r_1P_1+r_2P_2)\cap \Z^d)$
 is  a quasi-polynomial of degree $d$ (called the rational 
Ehrhart quasi-polynomial), {\it i.e.}, we have 
$$Q(P_1, P_2,;{\bf r}) = 
\sum_{(l_1, l_2)\in \Z^2_{\geq 0}, l_1+l_2 
\leq d}p_{l_1, l_2}({\bf r})r_1^{l_1}r_2^{l_2}$$
such that $(1)$ for some    $(\tau_1, \tau_2)\in \Q^2_{>0}$ we have 
 that
 $p_{l_1, l_2}(r_1, r_2) = 
p_{l_1, l_2}(r_1+\tau_1, r_2+\tau_2)$, for every 
$ {\bf r} = (r_1, r_2)\in \Q^2_{\geq 0}$
and (2) $p_{l_1,d-l_1}({\bf r})$ is independent of ${\bf r}$.

Here we prove:

\vspace{5pt}
\noindent~{\bf Theorem}~\ref{t1}\quad{\it There exists a decomposition $(0, \tau_1]\times (0,\tau_2] = \bigsqcup_{i=1}^nW_i$, 
 where $W_i$ are  locally closed subsets of $\R^2_{> 0}$  and, for each 
$(l_1, l_2)\in \Z^2_{\geq 0}$,
 there exists   a  set of polynomials
$\{f_{l_1, l_2}^i:W_i\longto \Q\}\}_i$
 such that 
$$p_{l_1,l_2}({\bf r}) = f_{l_1, l_2}^i({\bf r}),~~
\mbox{for every}~~{\bf r}\in W_i\cap \Q^2_{>0}.$$

 In particular,  for all ${\bf r}\in \Q^2_{>0}$, there exist constants
$C_{l_1,l_2}$ such that 
 $$|p_{l_1, l_2}({\bf r})|\leq C_{l_1, l_2},~\mbox{and}~ 
p_{l_1,d-l_1}({\bf r}) = C_{l_1, d-l_1}.$$}
 The proof of Theorem~\ref{t1} is a refinement  of the proof of
Theorem~1.3 of Henk-Linke ([HL]),
where they have proved (in our context) that  the
coefficients $p_{l_1,l_2}(-,-)$ are polynomials
on the interior of the $2$-cells in $\R^2$,
induced by the
hyperplane arrangement (given by the support functions of $P_1$ and $P_2$).
Since such cells do not cover $\R^2$ (or $(0, \tau_1]\times (0, \tau_2]$), and
the complement contains
 line segements,  the boundedness of the coefficients $p_{l_1,l_2}(-,-)$
 cannot be directly obtained from the result of [HL].

Note that their result was proved for Minkowski sums of any finite number of 
polytopes, which can also  be easily refined using similar methods 
(see Remark~\ref{rhl}). The similar result 
for rational Ehrhart  quasi-polynomial
for a single polytope was in an earlier paper of Linke (see Theorem~1.2,
Corollary~1.5 and Theorem~1.6 of [L]).   

Now  the uniform convergence of the 
sequence $\{g_n|_V\}_{n\in \N}$ to $g|_V$, for any compact set 
$V$ of $[0, \infty)\setminus v(\sP_D)$ is straightforward.
Using  the fact that the HK density function  $f_{R, {\bf m}}$
is compactly supported, continuous and (in this  toric case) is a piecewise polynomial function, 
we deduce that 
$$\int_0^{\infty} f_{R, {\bf m}}({\lfloor\lambda q\rfloor}/q)d\lambda =
\int_0^{\infty} f_{R, {\bf m}}(\lambda)d\lambda + O(1/q^2).$$
This and the similar approximation of the integral of the function 
$g_{R, {\bf m}}$  by the integrals of the functions $g_n$, as 
in the result stated above, implies that 
$\int g_{R, {\bf m}}(\lambda)d\lambda $ is the second coefficient of 
the HK function of $(R, {\bf m})$.

The paper is arranged as follows.

In Section~2, we recall Notations and known results about projective 
toric varieties, including a brief review of the density function, 
treated in detail in [MT].

In Section~3 deals with  the results about facets of the compact 
body ${\bar \sP_D}$, for a toric pair $(X, D)$.

Section~4 is an independent section on rational convex polytopes. Here we 
study the coefficients of 
the (multivariate) rational Ehrhart quasi-polynomial and prove they take only 
finitely many polynomial values. Now, for a polytope $P$ and $P_\lambda = 
P\cap \{z=\lambda\}$, we relate the coefficients of  the quasi polynomial 
$i(P_\lambda,n)$, for all $\lambda \in {\R_\geq 0}$ such that 
$\lambda n\in \Z$, 
 with the coefficients of a fixed rational Ehrhart quasi-polynomial of 
Minkowski sum of two polytopes. In particular, we get a uniform bound 
on  the coefficients of 
such $i(P_\lambda, n)$, which is important for the proof in Section~5.

In Section~5 we present the main result about the $\beta$-density function.

In Sections~6 we prove that the $\beta$~-density function is a 
multiplicative function for Segre products of toric pairs. Here we also compute 
the $\beta$-density functions for $(\P^2, -K)$, $(\F_a, cD_1+dD_2)$, where
$-K$ is the anticanonical divisor on $\P^2$, and where $D_1$ and $D_2$ are 
a natural basis for the 
$T$-Cartier divisors  of the Hirzebruch surface $\F_a$.

\section{Hilbert-Kunz density function on projective toric varieties}
Throughout the paper we work over an algebraically closed field $K$ with 
char $p > 0$ and follow the notations from \cite{Ful1993}.
Let  $N$ be  a lattice (which is isomorphic to 
$\Z^n$)
and let  $M = \text{Hom}(N, \Z)$ denote the dual lattice 
with a dual pairing $\langle\ , \rangle$.  
Let $T = \text{\mbox{Spec}}(K[M])$ be the torus with character lattice $M$. 
Let $X(\Delta)$ be a complete toric variety over $K$ with fan 
$\Delta \subset N_\R$.
We recall that the 
$T$-stable irreducible subvarieties of 
codimension $1$ of $X(\Delta)$ correspond to one dimensional cones 
(which are edges/rays of $\Delta $) of $\Delta$. If  
$\tau_1,\ldots, \tau_n$ denote the edges of the fan $\Delta$, then these 
divisors are the orbit closures $D_i= V(\tau_i)$. A base point free 
$T$-Cartier divisor 
$D=\sum_i a_iD_i$ (note that $a_i$ are integers) determines a  
convex lattice polytope in $M_{\mathbb{R}}$ defined by
\begin{equation}\label{*}
P_D=\{u\in M_{\mathbb{R}} \ | \ \langle u, v_i\rangle\geq -a_i 
~~\text{for all}\ 
i\ \}\end{equation}
and the induced embedding of $X(\Delta)$ in $\P^{r-1}$ is given by 
$$\phi=\phi_D: X(\Delta)\to\mathbb{P}^{r-1},\ 
\ x\mapsto ({\chi}^{u_1}(x):\cdots: {\chi}^{u_r}(x)),$$ 
where $P_D\cap M=\{u_1, u_2,\ldots, u_r\}$. 
For $(X(\Delta), D)$ and $P_D$  as above,  consider the cone $\sigma$ 
in $N\times \mathbb{Z}$ whose dual ${\sigma}^{\vee}$ is the cone over 
$P_D\times 1$ 
in $M\times \mathbb{Z}$. Then the affine 
variety $U_{\sigma}$ corresponding to the cone $\sigma$ 
is the affine cone of $X(\Delta)$ in  $\mathbb{A}^r_K$.
Therefore the homogeneous coordinate ring of $X(\Delta)$ (with respect 
to this embedding) is $K[{\chi}^{(u_1,1)},\ldots, {\chi}^{(u_r,1)}]$. 
Note that there is an isomorphism of graded rings (see Proposition 1.1.9, 
\cite{CLS2011})

\begin{equation}\label{**} \frac{K[Y_1,\ldots, Y_r]}{I}\simeq K[{\chi}^{(u_1,1)},
\ldots, {\chi}^{(u_r,1)}] =:R,\end{equation}

where, the kernel $I$ is generated by the binomials  of the form
$$Y_1^{a_1}Y_2^{a_2}\cdots Y_r^{a_r}-Y_1^{b_1}Y_2^{b_2}\cdots Y_r^{b_r}$$
where $a_1,\ldots, a_r, b_1,\ldots, b_r$ are nonnegative integers satisfying 
the equations 
$$a_1u_1+\cdots+a_ru_r=b_1u_1+\cdots+b_ru_r\ \ \text{and}\ \ 
a_1+\cdots+a_r=b_1+\cdots+b_r.$$

Note that due to this isomorphism, we can consider $R= K[S]$ as a standard 
graded ring, where $\deg~\chi^{(u_i, 1)} = 1$.

 Throughout this section we use the following

\begin{notations}\label{n1}
\begin{enumerate}\item 
A {\em toric pair} $(X, D)$ means $X$ is a projective toric variety over a field $K$ 
of charactersitic $p>0$,  with a very ample $T$-Cartier divisor $D$.

\item {\em The polytope  $P_D$ or $P_{X,D}$} is the  
 lattice polytope associated to the given  toric pair $(X, D)$ (as in (\ref{*})).

\item $f_{R, {\bf m}} = HKd_{R,{\bf m}}$ is the {\em HK density function} of $R$ with respect to 
the ideal ${\bf m}$, where 
  $R$ is the associated graded ring with the graded maximal ideal 
${\bf m}$ (as in (\ref{**})).

\item Let $C_D$ denote the convex rational polyhedral cone spanned by 
$P_D\times 1$ in $M_{\mathbb{R}}\times\mathbb{R}$.

\item Let 
$$\mathcal{P_D}=\{p\in C_D\ |\ p\notin (u,1) +C_D,
\text{for every}~~u\in P_D\cap M\}.$$
\item For a set $A\subset M_\mathbb{R}
\simeq \mathbb{R}^{d}$, we denote 
$$A\cap \{z=\lambda\} := A\cap \{({\bf x}, \lambda)\mid 
{\bf x}\in \R^{d-1}\}.$$ 
\end{enumerate}

\end{notations}

\begin{thm}(Theorem 1.1 in [T2])
\label{hkd}
Let $R$ be a standard graded Noetherian ring of dimension $d\geq 2$ 
over an algebraically closed field $K$ of characteristic $p > 0$, and let 
$I\subset R$ be a homogeneous ideal  such that $\ell(R/I) < \infty$.
For $n \in \mathbb{N}$ and $q = p^n$, let $f_n(R,I):[0,\infty)\longto [0,\infty)$
be defined as  
$$f_n(R, I)(x) = \frac{1}{q^{d-1}}\ell(R/I^{[q]})_{\lfloor xq\rfloor}.$$

Then
 $\{f_n(R, I)\}_n$ converges uniformly to a compactly supported 
continuous function\linebreak
$f_{R, I}:[0, \infty)\longto [0, \infty)$, where  $f_{R,I}(x)=
\lim_{n\to\infty} f_n(R,I)(x)$
and
$$e_{HK}(R, I)=\int_0^\infty f_{R, I}(x) \ dx.$$
\end{thm}

We recall Theorem~(1.1) from [MT]: 
\begin{thm}\label{thk}
Let $(X, D)$ be  a toric pair with associated ring  $(R, {\bf m})$, and 
 $P_D$, $C_D$ ${\sP_D}$ as 
in Notations~\ref{n1}. Then the Hilbert-Kunz 
density function of $(R,{\bf m})$ is given by the sectional volume of 
$\overline{\mathcal{P}}_D$, i.e.,  
$$f_{R, {\bf m}}(\lambda) = 
\textnormal{\mbox{rVol}}_{d-1}(\overline{\mathcal{P}}_D\cap \{z=\lambda\}),
\text{ for }\lambda\in [0,\infty).$$
Moreover, $f_{R,{\bf m}}$ is given by a piecewise polynomial function.
\end{thm}

\section{The boundary of $\sP_D$}
Recall that (Notations~\ref{n1}) associated to  a given toric pair $(X, D)$, we have 
a convex 
polytope $P_D$, a convex polyherdral cone $C_D$ and 
a bouned set $\sP_D\subset \R^d$.
In [MT], we had written a decomposition  $C_D = \cup_jF_j$, 
where $F_j's$ are 
$d$-dimensional cones such that, each $P_j := F_j\cap {\bar {\sP_D}}$  
 is a convex rational polytope and is a closure of $P_j':= 
F_j\cap {\sP_D}$. To study the boundary of $\sP_D$ we need a set of lemmas
about the facets of $P_j's$. We also assume without loss of generality that 
$d\geq 3$, as $d=2$ corresponds to $(\P^1, \sO_{\P^1}(n))$, for $n\geq 1$, 
which is easy to handle directly.

\begin{notations}\label{n3} 
\begin{enumerate}
\item $L(P_D) = P_D\cap M =$ the (finite) 
set of lattice points of $P_D$. 
\item For a convex polytope $Q$, let $v(Q) = \{\mbox{vertices of}~Q\}$ and
$F(Q) = \{\mbox{facets of}~Q\}$.
\item  $A(F) = \mbox{the affine hull  of}~ F$ in $\R^d$, for a 
set $F\subseteq \R^d$.
\item For a polytope $F$, $\dim~F := \dim~A(F)$.
\item $F_j \in \{d-\mbox{dimensional cones}\}$, which is the  
 closure of a  connected component of\linebreak 
$C_D\setminus \cup_{iu}H_{iu}$,
where the hyperplanes $H_{iu}$ are given by 
$$H_{iu} = ~~\mbox{the affine span of}~~\{(v_{ik},1), (u,1),
({\bf 0})\mid v_{ik}\in v(C_{0i}),~~u\in P_D\cap M\},$$
where $C_{0i} \in \{(d-3)~\mbox{dimensional faces of}~P_D\}$ and ${\bf 0}$ is the 
origin of $\R^d$.
\item For a subset $A\subseteq C_D$, we denote
\begin{enumerate}
\item $\partial_CA$ = boundary of $A$ in $C_D$  and 
\item $\partial(A)$ = the boundary of $A$ in $\R^d$.
\end{enumerate}
In particular $\partial_CA\subseteq \partial(A)$.
\item For $u\in L(P_D)$, let $C_u = (u,1)+C_D$ and let 
$$P_j'= F_j \cap\cap_{u\in L(P_D)}((u,1) + C_D)^c = 
F_j \cap\cap_{u\in L(P_D)}[C_D\setminus C_u],$$ which is a convex set
(Lemma~{4.5} of [MT]).
\item Let $P_j=\overline{F_j \cap\cap_{u\in L(P_D)}(C_D\setminus C_u)}$  
the convex rational polytope
which is the closure of $P_j'$ in $C_D$ (which equals the closure in $\R^d$). 

\end{enumerate}

Therefore  
$$\mathcal{P}_D=\cup_{j=1}^s P_j'\quad \mbox{and}\quad 
\overline{\mathcal{P}}_D=\cup_{j=1}^s P_j,$$
where $P_1, \ldots, P_s$ are distinct polytopes, whose interiors are disjoint.

Note that 
$$P_j= {\overline{F_j\setminus \cup_{u\in L(P_D)} C_u}} =
{\overline{\cap_{u\in L(P_D)} F_j
\setminus C_u}}.$$

\end{notations}

\begin{lemma}\label{dc}
For each $P_j$ as in Notations~\ref{n3}, we have
\begin{enumerate}
\item $$P_j=P_j'\sqcup \left(\cup_{u\in L(P_D)}\left\{ \partial_C(C_u)\cap P_j\right\}\right).$$
\item For any $u\in L(P_D)$, we have 
$\partial_C(C_u)\cap P_j = 
 \cup_{\{F'\mid F'\in F(C_u),~~F'\not\subseteq \partial(C_D)\}} F'\cap{P_j}$.
\end{enumerate}
\end{lemma}
\begin{proof}(1)~We only need to prove that 
$P_j\subseteq P_j'\bigsqcup_u\left\{\partial_C(C_u)\cap P_j\right\},$
as the other way inclusion is obvious.

Let $U= \cap _{u\in P_D\cap \Z^{d-1}} \left[C_D\setminus C_u\right]$. Then 
$P_j'= F_j\cap U$ and $P_j = {\overline{F_j\cap U}}$.  
It is easy to check that 
${\overline{F_j\cap U}} \subseteq (F_j\cap U)\cup (\partial_CU\cap 
{\overline{F_j\cap U}})$.
Now $$\partial_CU\subseteq \cup_u\partial_C(C_u) \implies 
(\partial_CU)\cap {\overline{F_j\cap U}}\subseteq 
\cup_{u\in L(P_D)}\left\{\partial_C(C_u)\cap P_j\right\}.$$
 This proves Assertion~(1).

\vspace{5pt}

\noindent(2)~We leave  this to the  reader.
This proves the lemma.\end{proof}

\begin{lemma}\label{l21}For any facet $F\in F(P_j)$ we have one 
and only one of the following possibilities:
\begin{enumerate}
\item $F\subset F_{j_i}$, for some facet $F_{j_i} \in F(F_j)$:  
(i)~In this case  $F= A(F_{j_i})\cap P_j = F_{j_i}\cap P_j$ and 
$$(ii)~~\dim~\left[A(F)\cap (\cup_{u\in L(P_D)}\partial_C(C_u)\cap P_j)\right]
\leq d-2.$$
\item $F\subset F_{u_\nu}$, for some facet $F_{u_\nu}\in F(C_u)$ and
$u\in L(P_D)$. In this case 
$F = P_j\cap F_{u_\nu} = P_j\cap A(F_{u_\nu})$, where $F_{u_\nu}\not\subseteq \partial(C_D)$.
\end{enumerate}
\end{lemma}
\begin{proof}Note $P_j' = F_j \cap_{u\in L(P_D)}\left[C_D\setminus 
C_u\right].$
Therefore 
$$F\subset \partial(P_j)\subseteq \partial(F_j)\cup\cup_{u\in 
L(P_D)}\partial[C_u] = 
\cup_{F_{j_i}\in F(F_j)}F_{j_i}\cup\cup_{u\in L(P_D)}
\cup_{F_{u_\nu}\in F(C_u)} F_{u_\nu}.$$
This implies $d-1 = \dim~F = \max\{\dim(F_{j_i}\cap F), 
\dim(F_{u_\nu}\cap F)\}_{j_i, u_\nu}.$  
Hence at least one of the sets, 
 $F_{j_i}\cap F$ or $F_{u_\nu}\cap F$, for some $j_i$ or $u_\nu$,  is of 
dimension $d-1$. 
 This implies either $A(F) = A(F_{j_i})$ or 
$A(F) = A(F_{u_\nu})$. 

\vspace{5pt}

\noindent{(1)}~~Let $F\in F(P_j)$ such that $F\subseteq A(F_{j_i})$, for some $j_i$.
Then $F = P_j\cap A(F_{j_i}) = 
P_j\cap F_j\cap A(F_{j_i}) = P_j\cap F_{j_i}$. In particular $F\subseteq F_{j_i}$. This 
proves Assertion~(1)~(i).

Suppose given $F\in F(P_j)$ such that $F\subseteq A(F_{j_i})$ and
$\dim~\left[A(F)\cap (\partial_C(C_u)\cap P_j)\right]= d-1$, for some $u\in L(P_D)$. Then 
$\dim~(A(F)\cap F_{u_\nu}\cap P_j) = d-1$, for some  $F_{u_\nu}\in F(C_u)$
 such that 
$F_{u_\nu}\not\subseteq \partial(C_D)$. Then $A(F) = A(F_{u_\nu}) = A(F_{j_i})$,
where $A(F_{j_i})$ is
a hyperplane passing through the origin ${\bf 0}$ of $\R^d$, 
and therefore $A(F)$ is a vector subspace of $\R^d$.
Also $F = (u, 1)+F'$, for some $F'\in F(C_D)$. Therefore 
$A(F) = (u,1) +A(F')$. Hence   $(u,1) + {\bf y} = {\bf 0}$, 
for some ${\bf y}\in A(F')$, 
which implies $A(F) = A(F')$. Therefore $F_{u_\nu}\subseteq A(F)\cap 
C_D = F'\subseteq
\partial(C_D)$, which is a contradiction.
This implies (1)~(ii) and hence the first assertion.

\vspace{5pt}

\noindent{(2)}~~We first prove the following 

\noindent{\bf Claim}~For any $F_j$ and for a facet $F''\in F(C_{u})$, where $u\in L(P_D)$, if 
$F''\cap F_j^o \neq\phi$ then $F''\cap F_j = A(F'')\cap F_j$, where 
$F_j^o = F_j\setminus \partial(F_j)$.

\vspace{5pt}

\noindent{\underline{Proof of the claim}}: Recall (see Lemma~4.5 of [MT]) 
that $F_j\setminus C_u$ is a convex set, for any $u\in L(P_D)$.

If $(A(F'')\cap F_j) 
\setminus F''\cap F_j\neq \phi$ then 
 there exists $x\in (\partial F'')\cap F_j^o$ and an open set  
(in $\R^d$) $B_x \subseteq F_j^o$ such that $B_x\cap C_{u_1} 
\cap F_j^o\neq \phi$. Hence there is another facet $F'\in F(C_{u_1})$ such that 
$F'\cap F_j^o\neq \phi$ ($x\in \partial(F')\cap 
F_j^o$). 

Note that $\dim(F'\cap F_j^o) = \dim~(F''\cap F_j^o) = d-1$ and 
$\dim~(F'\cap F'') \leq d-2$. Hence  we choose $x_1\in F''\cap F_j^o$ and 
$x_2 \in F'\cap F_j^o$ such that $x_1\neq x_2$. Then $tx_1+(1-t)x_2\subseteq F_j\cap 
C_{u_1}$, for $0\leq t\leq 1$.
Now we can also choose small enough nbhds (open in $\R^d$) 
$B_{x_1}$ and $B_{x_2}$ of $x_1$ and $x_2$ respectively, which are contained in 
$\mbox~F_j$.  Let $L$ be the affine line through  $x_1$ and $x_2$. Now, the line segment
of $L$ with end points 
$x_1'\in B_{x_1}\cap L\cap C_{u_1}^c$ and $x_2'\in B_{x_2}\cap 
L\cap C_{u_1}^c$, passes through $C_{u_1}$, which contradicts the 
convexity property of $F_j\setminus C_{u_1}$.
Hence the claim.

Suppose $F\in F(P_j)$ such that $F\not\subseteq A(F_{j_i})$, for any $F_{j_i}\in F(F_j)$.
Then there exists $F_{u_\nu}\in F(C_u)$, for some $u\in L(P_D)$, such that 
$\dim~(F\cap F_{u_\nu}) = d-1$. This implies $A(F) = A(F_{u_\nu})$ 
and therefore $F_{u_\nu}\not\subseteq A(F_{j_i})$, for any $F_{j_i}$.
On the other hand $F\cap F_{u_\nu}\subseteq F_j$. Hence
 $F_{u_\nu}
\cap F_j^o\neq \phi$, which implies, by the above claim that $A(F_{u_\nu})\cap F_j 
= F_{u_\nu}\cap F_j$.
Therefore 
$$F\subseteq  A(F_{u_\nu})\cap P_j\cap F_j = F_{u_\nu}\cap F_j\cap P_j = 
F_{u_\nu}\cap P_j \subseteq F.$$

Moreover, by definition $F\cap F_{u_\nu} \subset P_j\setminus P_j'$, therefore, 
by Lemma~\ref{dc}, $F_{u_\nu}\not\subseteq \partial(C_D)$.
 This proves the second 
assertion  and hence the lemma.\end{proof}

\begin{lemma}\label{l6}If $P_i\neq P_j$ then  $\dim(P_i\cap P_j) \leq d-1$. 
Moreover, if  $\dim(P_i\cap P_j) = d-1$ then $P_i\cap P_j \in 
F(P_i) \cap F(P_j)$, that is, $P_i$ and $P_j$ meet along a common facet.
\end{lemma}
\begin{proof}We know $P_i\cap P_j \subseteq F_i\cap F_j$, where 
$\dim~(F_i\cap F_j) \leq d-1$. If
$\dim~(P_i\cap P_j) = d-1$ then $F_i\cap F_j\in F(F_i)\cap F(F_j)$. 
Let $F = F_i\cap F_j$ then 
$F_i\cap A(F) = F_j \cap A(F) = F$. But $P_i = {\overline{F_i\cap U}}$ and 
$P_j = {\overline{F_j\cap U}}$. Therefore $P_i\cap A(F) = 
{\overline{F_i\cap U}}\cap A(F)=  {\overline{F_i\cap U\cap A(F)}} 
= {\overline{F\cap U}} = 
 P_j\cap A(F)$.
This proves the lemma.
\end{proof}

\begin{lemma}\label{l7}$$ \begin{array}{lcl} 
(1)~~~\partial(\sP_D) & = & 
\cup_{\{F\in F(P_j)\mid F\neq P_i\cap P_j\}} F.~~\mbox{In particular}\\\\
(2)~~~\partial(\sP_D) & = & \bigcup_{\{ F\in F(C_D)\}}F\cap {\bar \sP_D} \cup 
\bigcup_{\{F\in F(C_u), u\in L(P_D)\}}F\cap {\bar \sP_D}.
\end{array}$$
\end{lemma}
\begin{proof}Note $\partial(\sP_D) = \partial({\overline \sP_D})
\subseteq \cup_{F\in F(P_j)}F$.
Moreover $\cup_{j,\nu}E_{j\nu}\subseteq \partial(\sP_D)$, where 
\begin{equation}\label{be}\{E_{j\nu}\}=  \{F\in F(P_j) \mid F\subseteq F',~F'\in F(C_u),~u\in 
L(P_D)\}\subseteq 
{\overline \sP_D}\cap {\overline {\sP_D^c}}\subseteq \partial(\sP_D).
\end{equation}
Let $V = \cup_F\{F\mid F\in F(P_j),~~ F\neq P_i\cap P_j,~\mbox{for any}~ j\}$. 
Then $\cup_{j,\nu}E_{j\nu}\subseteq V$ (see Lemma~\ref{l21}).
Therefore $\partial(\sP_D)\setminus V\subseteq \cup_F\{F\mid F= 
P_i\cap P_j\}$.
If $\partial(\sP_D)\setminus V\neq \phi$ then there exists 
$B_{d-1} \subseteq \partial \sP_D\cap P_i\cap P_j$, for some $P_i\neq P_j$.
 Therefore $P_i\setminus V$ and $P_j\setminus V$ are nonempty open sets 
of $P_i$ and $P_j$ respectively. Hence we can choose an open set 
$B_d$ such that 
$$B_d = [B_{d}\cap (P_i\setminus V)]\cup [B_{d}\cap (P_j\setminus V)]\cup 
[B_{d}\cap (P_i\cap P_j)],$$ where $B_d\cap P_i\cap P_j \subseteq B_{d-1}$. 
Since $B_{d-1}\subseteq P_i\cap P_j$, we have $B_{d-1}\cap {\overline \sP_D}^c 
= \phi $, which implies $B_{d-1}\cap \partial(\sP_D) = \phi$, 
hence a contradiction.
Therefore $\partial(\sP_D) = V$. This proves Assertion~(1).
Now 
$$
\partial(\sP_D)  =  \bigcup_{\{F\in F(P_j),~F\subseteq A(F_j),~F\neq 
P_i\cap P_j\}}F\cup \bigcup_{j,\nu}\{E_{j\nu}\}
 =  \bigcup_{\{ F\in F(C_D)\}}F\cap {\bar \sP_D} \cup 
\bigcup_{\{F\in F(C_u), u\in L(P_D)\}}F\cap {\bar \sP_D}.$$
This proves the lemma.
\end{proof}

\vspace{5pt}

\begin{notations}\label{ndn}
In the rest of the paper, 
for a bounded set $Q\subset \R^d$ and for $n, m\in\N$,  
we define 
\begin{equation}\label{dn}i(Q, n, m):= 
\#(nQ\cap \{z = m\}\cap \mathbb{Z}^{d}),
\end{equation}
where $z$ is the $d^{th}$ coordinate function on $\R^d$.
\end{notations}

\begin{rmk}\label{r1}
From lemmas~\ref{dc} and \ref{l2}, it follows that  
$$P_j=P_j'\bigsqcup (\cup_{\gamma} E_{j\gamma}),~~\mbox{where}~~ 
\{E_{j\gamma}\} = \{F\in F(P_j)\}_{\{F \subseteq
 F',~F'\in F(C_u),~u\in L(P_D)\}}$$
and $E_{j\nu}\not\subseteq \partial~(C_D)$.
Note,  $E_{j\gamma}\cap P_i' = \phi$, for every $i$, as 
$P_i'\subseteq \cup_{u\in L(P_D)} C_D\setminus C_u$.
In particular, 
$$i(\sP_{D}, n, m) = i(\cup_jP_j', n, m) = i(\cup_jP_j, n, m) - 
i(\cup_{j\gamma}E_{j\gamma}, n, m).$$
Therefore, we have 
\begin{equation}\label{***}
 i(\mathcal{P}_D, n, m) =
\sum_ji(P_j, n, m)-\sum_{j< k}i(P_j\cap P_k, n, m)
- \sum_{j,\gamma}i(E_{j\gamma}, n, m)+
\sum_{\alpha\in I_1}\epsilon_\alpha i(Q'_{\alpha}, n, m),
\end{equation}
where $\{Q'_\alpha\}_{\alpha\in I_1}$ runs over a certain finite set 
of polytopes of dimension $\leq d-1$: either
$Q'_\alpha =  P_{j_1}\cap \cdots\cap P_{j_l}$, for distinct $P_{j_i}'s$, where
$l\geq 3$,
or
$Q'_\alpha = E_{j_1\gamma_1}\cap \cdots\cap E_{j_l\gamma_l}$, 
for distinct $E_{j_i\gamma_i}'s$, and $l\geq 2$.
Note that $\epsilon_\alpha \in \{1, -1\}$, depending on $\alpha \in I_1$.
 \end{rmk}

\begin{lemma}\label{l8}Let $Q$ be a convex polytope such that 
$Q\subseteq F$, for some facet $F\in F(P_j)$, where $1\leq j\leq s$.
Then $\dim(Q\cap\{z=\lambda\})\leq d-2$, 
for all $\lambda\in \R_{\geq 0}$. Moreover
\begin{enumerate}
\item if $\dim~Q\leq d-2$ then $\dim~Q\cap \{z=\lambda\} 
= d-2$,  at the most for one $\lambda \in \R_{\geq 0}$.
\item If $dim~Q = d-1$ then 
 $\left[A(Q)\cap \{z= m\}\right]\cap \Z^d \neq \phi$, 
for every $m\in \Z$.
\item If $\dim~Q\cap \{z=\lambda\} = d-2$, for some $\lambda \in \R_{> 0}$
then $[A(nQ)\cap\{z =m\}]\cap \Z^d \neq \phi$, whenever $n, m\in \Z_{> 0}$
such that $m/n =\lambda$.
\end{enumerate}
\end{lemma}
\begin{proof} By definition, $Q\subseteq F$, for some $F
\in F(F_j)$ or, for some $F\in F(C_u)$
and   $u\in L(P_D)$.
But such hyperplanes  are  transversal to the 
hyperplane $\{z=0\}$.
Hence $\dim~Q\cap \{z=\lambda\} \leq d-2$, for every 
$\lambda\in \R_{\geq 0}$.

\noindent\quad (1) Suppose $\dim(Q\cap \{z=\lambda_0\}) = d-2$, 
 for some $\lambda_0\in \R_{\geq 0}$. 
Then 
$$A(Q) = A(Q\cap \{z=\lambda_0\}) = 
A(Q)\cap \{z=\lambda_0\}.$$ 
Therefore 
$Q \subseteq A(Q\cap \{z=\lambda_0\})$ and 
 $Q\cap \{z=\lambda\} = \phi$, for $\lambda \neq \lambda_0$.
Hence 
 $\dim~Q\cap \{z=\lambda\} = d-2$, at the most at one point.

\vspace{5pt}

\noindent\quad (2)
By  Lemma~\ref{l21}, we have 
$Q \subseteq  A(F)$, where $F\in F(F_j)$ or $\in F(C_u)$.

\noindent{Case}~(1)\quad Let $F$ be a facet of $F_j$ 
for some $F_j$. Then $F\subseteq H_{iu}$ for some hyperplane $H_{iu}$ 
(as given in Notations~\ref{n3})~(5)) and 
 $A(Q) = A(F) = H_{iu}$.
Hence, for $m\in \Z$, we have 
$A(Q)\cap \{z=m\} = H_{iu}\cap \{z = m\}$, where it is easy to check that 
$m(u,1)\in H_{iu}\cap \{z =m\}\cap \Z^d$.

\noindent{Case}~(2)\quad
If $F$ is a facet of $(u, 1)+C_D$ then  $F= (u,1)+F'$, for 
some facet $F'$ of $C_D$. Now $F'$ is a
 cone over a facet $F'' $ of $P_D$.
Hence there exist a subset of vertices $\{u_j\}\subset \Z^{d-1}$ of $P_D$ 
such that 
  $A(Q) = A(F) = (u,1)+\sum_j\alpha_j(u_j,1)\mid \alpha_j\in \R\}.$
Now it is easy to check that $(u,1)+(m-1)(u_j,1) \in A(F)\cap \{z=m\}\cap \Z^d$.

\noindent\quad (3) With the notations as 
in (2), we have 
 $A(nQ\cap \{z= n\lambda\}) = A(nQ)\cap \{z= n\lambda\} = 
A(nF)\cap \{z= n\lambda\}$. Now, for 
$F\in F(F_i)$, 
one can check that $n\lambda(u,1)\in A(nF)\cap \{z= n\lambda\}\cap 
\Z^d$, and  for $F\in F(C_u)$ for some $u\in L(P_D)$ one can check
$n(u,1)+(\lambda-1)n(u_j,1)\in A(nF)\cap \{z= n\lambda\}\cap 
\Z^d$. 
This completes the proof of the lemma.

\end{proof}

\section{Ehrhart quasi-polynomial for rational convex polytope}

In this section, we mainly deal with the Ehrhart's theory of lattice 
points inside rational convex polytopes. Recall that, if $\Z^d$ is the 
integral lattice in $d$-dimensional euclidean space $\R^d$, then a 
convex polytope $P\subset \R^d$ is called integral (rational), if all 
its vertices have integral (rational) coordinates.

\begin{defn}\label{d31} For a rational polytope $P$, 
the smallest number $\rho\in \Q_{>0}$ such that $\rho P$ is 
an integral polytope is called the rational denominator of $P$, and 
is denote by $\tau(P)$. Furthermore, 
the rational $i$-index  $\tau_i(P)$ of $P$ is the 
smallest number $\rho \in \Q_{>0}$ such that 
for each $i$-dimensional face $F$ of $P$ the affine space 
$\rho\hspace{.05cm}A(F)$ contains integral points. 
Here $A(F)$ denotes the affine hull of $F$.
\end{defn}

The following classical  result is due to Ehrhart([Eh]) 
and McMullen(\cite{McM78}).

\begin{thm}\label{ehrhart}
Let $P \subset \R^d$ be a rational polytope. Then
$i(P, -):\Z_{\geq 0} \to \Z_{\geq 0}$, given by 
$$i(P,n):=\#(nP \cap\Z^d) = 
\sum_{i=0}^{\dim(P)} C_i(P,n)n^i,\ \ \text{for}\ \ n\in\Z_{\geq 0},$$
  is a quasi-polynomial of 
degree $\dim{P}$, {\it i.e.},  for every $i$, the coefficient $C_i(P, n)$ 
is periodic in $n$ of period 
$\tau_i(P)$, and 
$C_{\dim{P}}(P, n)$ is not identically zero (in fact is 
$= \mbox{rVol}_{\dim{P}}(P)$, if $A(P)$ contains an integral point).

Moreover 
if $P^\circ$ denotes the interior of $P$ in the affine span of $P$, 
then $i(P^\circ, n) 
=\#(nP^\circ\cap \Z^{d})= (-1)^{\dim(P)} i(P, -n)$.
  In particular, 
$$A(F)\cap \Z^d \neq\phi,~\mbox{for every}~F\in F(P) \implies 
C_{\dim(P)-1}(P, n) = (1/2)\sum_{F\in F(P)} \mbox{rVol}_{\dim(P)-1}(F).$$
\end{thm}

Here note that, since $i(P,n)$ is a quasi-polynomial, it can be defined for all 
$n \in \Z$. 

McMullen has generalised Theorem \ref{ehrhart} for the rational 
Minkowski sum of finitely many polytopes $P_1,\dots, P_k\subset \R^d$. 
Throughout this section we use the following notations and definition from the literature. 

\begin{notations}
\begin{enumerate}\item For ${\textbf{r}}=(r_1,\ldots, r_k)\in\R^k$ and 
$\textbf{\textit{l}}=(l_1,\ldots,l_k)\in \Z_{\geq 0}^k$, we denote 
$\prod_{i=1}^{k}r_i^{l_i}$ by ${\textbf{r}}^{\textbf{\textit{l}}}$ and 
$\sum_il_i$ by $|\textbf{\textit{l}}|$.
\item  The Hadamard product 
${\textbf{r}} \odot{\textbf{s}}$ of two rational vectors 
$\bf{r},{\textbf{s}}\in\Q^k$ is the coordinate-wise product 
${\textbf{r}}\odot{\textbf{s}}=(r_1s_1,\ldots, r_ks_k)$. 
\end{enumerate}
\end{notations}

With these notations McMullen's result (see comments on page~2 of [HL]) 
 on the Ehrhart quasi polynomial 
of a  Minkowski sum of rational polytopes can be stated as follows.

\begin{thm}\label{McMullen}
Let $P_1,\ldots,P_k \subset \R^d$ be rational polytopes. Then 
the function 
$Q(P_1,\ldots,P_k, -) : \Q^k_{\geq 0} \to \N$ given by 
$$Q(P_1,\ldots,P_k, {\bf r})=
\#(\sum_ir_iP_i\cap \Z^d),~ 
\text{for}\ {\textbf{r}}=(r_1,\ldots, r_k)\in\Q^k_{\geq 0}$$
is a rational 
quasi-polynomial of degree $\dim(P_1 +\cdots+P_k)$ with period 
${\bf \tau}=(\tau(P_1),\ldots,\tau(P_k))$, {\it i.e.},
$Q(P_1,\ldots,P_k, {\bf r})= \sum_{|\textbf{\textit{l}}|\leq d} 
p_{\textbf{\textit{l}}}({\textbf{r}}){\textbf{r}}^{\textbf{\textit{l}}},$ 
where
$p_{\textbf{\textit{l}}} : \Q^k_{\geq 0} \to \Q$ is a periodic function 
with period $\tau_i = \tau(P_i)$ in the 
$i$-th argument, $i=1,\ldots, k$, and $p_{\textbf{\textit{l}}}({\bf r})$ 
is non-zero positive constant
for some ${\textbf{\textit{l}}}\in \Z^k_{\geq 0}$ with 
$|{\textbf{\textit{l}}}| = \dim(P_1+\cdots+P_k)$.
\end{thm}
\begin{proof}
See \cite{McM78}, Theorem 7.
\end{proof}

\begin{defn}
$Q(P_1,\ldots,P_k, -) $ is called the {\em rational Ehrhart quasi-polynomial} 
of $P_1,\ldots, P_k$, and the $\textbf{\textit{l}}$-th coefficient of 
$Q(P_1,\ldots,P_k, -) $ is denoted by $Q_{\textbf{\textit{l}}}(P_1,\ldots,P_k, -)$. 
\end{defn}

In 2011, Linke  has proved (see Theorem~1.2, Corollary~1.5 and 
Theorem~1.6 of [L]) the following result 
about the coefficients of rational Ehrhart quasi-polynomial of a 
rational polytope.

 \begin{thm}\label{tl}
 Let $P \subset \R^d$ be a rational polytope of dimension $d$ with 
rational Ehrhart quasi-polynomial 
$$i(P,r):=\#(rP \cap\Z^d)=
\sum_{i=0}^{\dim(P)}C_i(P,r)r^i,~\mbox{where}~ r \in \Q_{\geq 0}.$$

Then (1)~~there exist $0=r_0<r_1<\cdots <r_l=\tau(P)$, such that 
$C_i(P,-)$ is a polynomial of degree $d-i$ on $(r_{m-1}, r_m)$, for each 
$m=1,\ldots, l$ and $i=0,\ldots , d$.  (2)~~The reciprocity 
theorem is true for rational dilates and for all dimension, i.e. for all 
$r\in \Q_{\geq 0}$,
$i(P, r)=(-1)^{\dim(P)}i(P, -r)$. In particular $C_d(P, r) = \mbox{Vol}_d(P)$, for all 
$r\in \Q_{>0}$, and, in addition, if $C_{d-1}(P, r)$ is independent of 
$r>0$ then 
$C_{d-1}(P, r) = (1/2)\sum_{F\in F(P)}\mbox{rVol}_{d-1}F$.
\end{thm}
Later, the above theorem was generalised for Minkowski sum of polytopes 
by Henk and Linke in their paper \cite{HenkLinke2015} (Theorem~1.3).

We recall briefly some important points relevant 
to the statement of Theorem~1.3 of [HL].

For a polytope $P \subset \R^d$, let $h(P,-) : \R^d \to \R$ be its support 
function, i.e., $h(P,v) = \text{max}\{\langle v, x\rangle : x \in P\}$. 
A hyperplane
$H(P,v) := \{x \in {\R}^d\ \arrowvert\ \langle x,v\rangle = h(P,v)\},$
for $v\in {\R}^d\setminus \{0\}$ is called a \textit{supporting hyperplane} of $P$. 
If $P$ is full dimensional, i.e., $\dim(P)=d$, then each facet $F$ of 
$P$ is given by a unique supporting hyperplane 
$H_F=\{x\in {\R}^d\mid \langle x, a_F\rangle = b_F\}$,
where $(a_F, b_F)\in {\R}^d\times\R$ is unique up to multiplication 
by a positive real number. Let 
$H_F^{-1}=\{x\in {\R}^d\mid \langle x, a_F\rangle \leq b_F\}$,   Let $P\subset {\R}^d$ be a full dimensional 
lattice polytope; the hyperplane representation of $P$ is 
$$P=\bigcap_{F\in F(P)}H_F^{-}=\bigcap_{F\in F(P)}\{x\in {\R}^d\mid \langle x, 
a_F\rangle \leq b_F\}$$
where the intersection runs over all facets $F$ of $P$. We call $a_F/||a_F||$
the {\em outer unit normal} of the facet $F$ of $P$.

Let $P_1,\ldots,P_k \subset\R^d$ be rational polytopes with 
$\dim(P_1+\cdots+P_k)= d$. Let $\textbf{v}_1,\ldots, \textbf{v}_m \in \Z^d$, be the 
outer 
normals of the facets of the rational polytope $P_1+\ldots+P_k.$ Observe 
that for all $\textbf{r} \in\R^k_{> 0}$ the facets of the polytope 
$r_1P_1+\ldots+r_kP_k$ have the same outer normals 
$\textbf{v}_1,\ldots, \textbf{v}_m$. 
For details about support function and face decomposition of Minkowski 
sum, see  \cite{Sch2013}.
Now, for $\textbf{r} \in \R^k_{> 0}$ and $\textbf{z} \in\Z^d$ we 
know $\textbf{z}\in \sum_{i=1}^{k}r_iP_i$ if and only if 
$\langle \textbf{z}, \textbf{v}_j\rangle \leq \sum_{i=1}^{k}r_ih(P_i,{\bf v}_j)$ 
for $1\leq j \leq m$. Thus $Q(P_1,...,P_k,\textbf{r})$ is a constant 
function on the interior of the $k$-dimensional cells induced by the 
hyperplane arrangement
\begin{eqnarray}
\label{ehyp}
\left\{ \big\{{\textbf{r}}\in \R^k_{> 0} :\sum_{i=1}^{k}r_ih(P_i,{\bf v}_j)=
\langle \textbf{z}, \textbf{v}_j\rangle\big\} : {\textbf z}\in\Z^d, 
j=1,\ldots,s\right\}.
\end{eqnarray}
Let $S$ be the interior of a fixed $k$-dimensional cell given by 
this section. 
Then $Q(P_1,\ldots,P_k,-)$ is constant on $S$.

The result of Henk-Linke (Theorem~1.3) can be stated as follows
\begin{thm}\label{t3}
Let $P_1,\ldots,P_k \subset\R^d$ be rational polytopes with 
$\dim(P_1+\cdots+P_k)= d$ and let $\textbf{\textit{l}} \in\Z_{\geq 0}^k$ 
with $|\textbf{\textit{l}}|\leq d$. Then $Q_{\textbf{\textit{l}}}(P_1,\ldots,P_k,-)$ is a 
piecewise polynomial function of degree at most 
$d-|\textbf{\textit{l}}|$ on open $k$-cells given by 
the hyperplane arrangements as in (\ref{ehyp}).
\end{thm}

Here we  consider 
the case with $k=2$. By Theorem \ref{McMullen}, $Q(P_1, P_2, -)$ is a 
quasi-polynomial of degree dim$(P_1+P_2)$ with period $\tau=(\tau_1, \tau_2)$. 
The hyperplane arrangement in (\ref{ehyp}) can be rewritten as
\begin{equation}\label{3*}
\left\{ \big\{{\textbf{r}}\in \R^2_{> 0} : r_1h(P_1,{\bf v}_j)+r_2h(P_2, {\bf v}_j)=
\langle \textbf{z}, \textbf{v}_j\rangle\big\} :{\textbf z}\in\Z^d, 
j=1,\ldots,s\right\}
\end{equation}
and $Q(P_1, P_2, -)$ is constant on each open $2$-cell in the complement of 
these lines.
 
\begin{notations}\label{n2}
\begin{enumerate}
\item For ${\textbf z}\in \Z^d,$ and $1\leq j \leq m$, we denote 
 the line 
$$ L_{j}(\textbf{z}) =  \big\{{\textbf{r}}\in \R^2 : 
r_1h(P_1,{\bf v}_j)+r_2h(P_2,{\bf v}_j)=\langle \textbf{z}, \textbf{v}_j\rangle\big\}.$$  

 the positive halfspace 
 $L_{j}(\textbf{z})^+ = 
 \big\{{\textbf{r}}\in \R^2 : 
r_1h(P_1,{\bf v}_j)+r_2h(P_2,{\bf v}_j)\geq\langle \textbf{z}, 
\textbf{v}_j\rangle\big\}$
and the positive open halfspace 
 $L_{j}(\textbf{z})^{+\circ}$ the interior of 
$L_{j}(\textbf{z})^+$, {\em i.e.}, 
$$L_{j}(\textbf{z})^{+\circ} = \big\{{\textbf{r}}\in \R^2 : 
r_1h(P_1,{\bf v}_j)+r_2h(P_2,{\bf v}_j)>\langle \textbf{z}, \textbf{v}_j\rangle\big\}.$$ 
Similarly one defines $L_{j}(\textbf{z})^-$ and $L_{j}(\textbf{z})^{-\circ}$ 
for ${\textbf z}\in \Z^d$ and for $j=1,\ldots , s.$

\item Denote the period rectangle $T = (0, \tau_1]\times (0, \tau_2]$, where 
$(\tau_1, \tau_2) = (\tau(P_1), \tau(P_2))\in \Q^2_{>0}$.
\item Note that for each $j\in \{1,\ldots, m\}$, there can be only finitely 
many $L_j(\textbf{z})$ 
intersecting the period rectangle $T$, as ${\bf v}_j\in \Z^d$.

\item $\R^2_{>0}$ is the disjoint union of locally closed 
sets, namely 
\begin{equation}\label{e3}\R^2_{> 0} = (\bigcup_{S\in {\tilde C_P}}S)\cup 
(\bigcup_{I\in {\tilde I_P}}I)\cup T_0,\end{equation}
where ${\tilde C_P}$ = the set of open $2$-cells obtained by the hyperplane 
arrangement as given in (\ref{3*}), 
 the set $$T_0 = \{L_j({\bf z}) \cap L_i({\bf z}')\cap \R^2_{> 0}\mid L_j({\bf z}) 
\neq L_i({\bf z}'),
\quad\mbox{for all}~z, z'\in \Z^d~\mbox{and}~1\leq i, j\leq m\}$$
is a discrete  set of points and 
$${\tilde I_P} = \mbox{the connected components of}~~ 
\R^2_{> 0} \setminus (\bigcup_{S\in {\tilde C_P}}S)\cup T_0 $$
is the set of open intervals.
In particular, for any $I\in {\tilde I_P}$, there exists a unique line  
$L_{j_0}({\bf z}_0)$ such that $I$ is 
 a  conncted component of 
$$L_{j_0}({\bf z}_0)\setminus \{L_{j_0}({\bf z}_0)\cap L_j({\bf z})\mid  
L_{j_0}({\bf z}_0) \neq L_j({\bf z}),~~1\leq j\leq m,~~{\bf z}\in 
\Z^d\}\cap \R^2_{>0}.$$

\end{enumerate}
\end{notations}

\begin{defn}\label{d1}For given $I \in {\tilde I_P}$, we associate 
(the unique) $S_I \in {\tilde C_P}$ as follows: By definition 
$I \subset L_j({\bf z})$, for a unique line $L_j({\bf z})$ (as in Notation~\ref{n2})~in $\R^2$.
Then  $S_I$ is  the  unique cell in ${\tilde C_P}$ such 
that $I\subset {\bar S_I}$, the 
 closure of $S_I$ in $\R^2_{>0}$ and $S_I\subset L_j({\bf z})^{+o}$. 
\end{defn}

\begin{lemma}\label{l3}Given $(I, S_I)\in {\tilde I_P}\times {\tilde C_P}$,  we have 
$$S_I\subset L_{j_l}({\bf z}_{i_l})^+ \iff I\subset L_{j_l}({\bf z}_{i_l})^+.$$
Moreover, if $L_{j_0}({\bf z}_{i_0})$  is the line containing $I$,
then for 
$L_{j_l}({\bf z}_{i_l}) \neq L_{j_0}({\bf z}_{i_0})$, we have 
$$ S_I \subset L_{j_l}({\bf z}_{i_l})^+ \iff S_I \subset L_{j_l}({\bf z}_{i_l})^{+o} 
\iff I \subset L_{j_l}({\bf z}_{i_l})^{+o} \iff I\subset L_{j_l}({\bf z}_{i_l})^+.$$
\end{lemma}
\begin{proof} It is easy to check.
\end{proof}

\begin{lemma}\label{l4} Given $S\in {\tilde C_P}$, 
$Q(P_1,P_2, {\textbf{s}}) = \mbox{constant\quad for all}\quad
{\textbf{s}}\in S.$

Given  $(I, S_I)\in {\tilde I_P}\times {\tilde C_P}$
$$Q(P_1,P_2, {\textbf{s}}) = \mbox{constant for all}\quad
{\textbf{s}}\in S_I\cup I.$$
\end{lemma}
\begin{proof}It is easy to check.\end{proof}

\begin{lemma}\label{l5}Let ${\bf u} = (u_1, u_2)\in \Z^2_{\geq 0}$ and let
 $T = (0\times \tau_1]\times
(0\times \tau_2]$ be as in Notations~\ref{n2}.
Then, for given 
\begin{enumerate}
\item  
$S\in {\tilde C_P}$, we have  
$Q(P_1,P_2, {\textbf{s}}) = \mbox{constant, for all}\quad
{\textbf{s}}\in S\cap T + \textbf{u}\odot {\bf\tau}$ and, for given
\item   $(I, S_I)\in {\tilde I_P}\times {\tilde C_P}$,  we have 
$Q(P_1,P_2, {\textbf{s}}) = \mbox{constant, for all}\quad
{\textbf{s}}\in (S_I\cup I)\cap T + \textbf{u}\odot {\bf\tau}.$
\end{enumerate}
\end{lemma}
\begin{proof}For $(u_1, u_2)\in \Z^2_{\geq 0}$, the polytope
 $u_1\tau_1P_1+u_2\tau_2P_2$ is an integral polytope. Therefore, for every 
facet $F_j$ of (this polytope) with the outer normal ${\bf v}_j$, we can choose 
${\tilde {\bf z}_j}\in F_j\cap \Z^d$ such that 
$h(u_1\tau_1P_1+u_2\tau_2P_2, {\bf v}_j) = \langle{{\tilde {\bf z}_j}, 
{\bf v}_j}\rangle $.
Now, it is easy check that, for every $1\leq j\leq m$, 
$$L_j({\bf z}+{\tilde {\bf z}_j}) = L_j({\bf z}) + {\bf u}\odot {\bf \tau},\mbox{}\quad 
L_j({\bf z}+{\tilde {\bf z}_j})^{+o} = L_j({\bf z})^{+o} + {\bf u}\odot {\bf \tau} 
\quad\mbox{and}\quad L_j({\bf z}+{\tilde {\bf z}_j})^{-o} = L_j({\bf z})^{-o} + 
{\bf u}\odot {\bf \tau}.$$

A given $S\in {\tilde C_P}$  can be written as
$$S = \left[\cap_{\mu=0}^{s_1}L_{j_\mu}
(\textbf{z}_{i_\mu})^{+\circ}\right]\cap
\left[\cap_{\nu=1}^{s_2}L_{l_\nu}(\textbf{z}_{k_\nu})^{-\circ}\right]\cap 
\R^2_{>0},$$
for some 
$\textbf{z}_{i_\mu}, \textbf{z}_{k_\nu}\in \Z^d$ and 
$1\leq j_{\mu}, l_{\nu}\leq m$.

Therefore 
$$S + {\bf u}\odot {\bf \tau} = \left[\cap_{\mu=0}^{s_1}L_{j_\mu}
(\textbf{z}_{i_\mu})^{+\circ} + {\bf u}\odot {\bf \tau}\right]\cap
\left[\cap_{\nu=1}^{s_2}L_{l_\nu}(\textbf{z}_{k_\nu})^{-\circ} + {\bf u}\odot {\bf \tau}\right]\cap 
\left[\R^2_{>0}+ {\bf u}\odot {\bf \tau}\right],$$
$$ = \left[\cap_{\mu=0}^{s_1}L_{j_\mu}
(\textbf{z}_{i_\mu} + {\bf z}_{j_{\mu}})^{+\circ} \right]\cap
\left[\cap_{\nu=1}^{s_2}L_{l_\nu}(\textbf{z}_{k_\nu}+
{\bf z}_{l_{\nu}})^{-\circ}\right]\cap 
\left[\R^2_{>0}+ {\bf u}\odot {\bf \tau}\right].$$
Note that, for any ${\bf z}\in \Z^d$ and $1\leq j
\leq m$, if $S\in {\tilde C_P}$ then
  we have 
$L_j({\bf z}-{\tilde {\bf z}_j})\cap S = \phi$, since
$L_j({\bf z}-{\tilde {\bf z}_j})$ is one of the lines in the hyperplane 
arrangement~\ref{ehyp}.  
This implies 
$L_j({\bf z})\cap (S + {\bf u}\odot{\bf \tau}) = 
(L_j({\bf z}-{\tilde {\bf z}_j})
\cap S)+{\bf u}\odot {\bf \tau} = \phi$.
Hence $S + {\bf u}\odot {\bf \tau} \subseteq S_1$, for some $S_1\in {\tilde C_P}$.

Therefore, by Lemma~\ref{l4}, 
$$Q(P_1,P_2, {\textbf{s}}) = \mbox{constant, for all}\quad
{\textbf{s}}\in (S + \textbf{u}\odot {\bf\tau})\cap 
(T+\textbf{u}\odot {\bf\tau}) = (S\cap T) + \textbf{u}\odot {\bf\tau}.$$

This proves Assertion~(1).

Note that $I\in {\tilde I_P}$ if and only if  $I \subset L_{j_0}({\bf z}_0)$, some 
$1\leq j_0\leq m$ and ${\bf z}_0\in \Z^d$, 
such that $I$ is a connected component of 
$$ (L_{j_0}({\bf z}_0)\setminus \{L_{j_0}({\bf z}_0)\cap 
L_j({\bf z})\mid {\bf z}\in \Z^d,~~~1\leq j\leq m~\mbox{and}~
 L_j({\bf z})\neq L_{j_0}({\bf z}_0)\})\cap \R^2_{>0}.$$
This implies that $I+ {\bf u}\odot{\bf \tau} $ is a connected component of 
$$\{L_{j_0}({\bf z}_0+{\tilde {\bf z}_{j_0}})\setminus \{L_{j_0}({\bf z}_0+
{\tilde {\bf z}_{j_0}})\cap 
L_j({\bf z}+{\tilde {\bf z}_j})\mid {\bf z}\in \Z^d,~~~1
\leq j\leq m,~\mbox{and}~L_{j_0}({\bf z}_0+
{\tilde {\bf z}_{j_0}})\neq
L_j({\bf z}+{\tilde {\bf z}_j})\}\cap \R^2_{>0}.$$
Hence $I+{\bf u}\odot{\bf \tau}\in {\tilde I_P}$.
One can easily check (from Lemma~\ref{l3}) that 
$S_I+{\bf u}\odot 
{\bf \tau} = S_{I+ {\bf u}\odot{\bf \tau}}$. Hence,
$$(S_{I+{\bf u}\odot{\bf \tau}} \cup (I+{\bf u}\odot{\bf \tau}))\cap (T 
+\textbf{u}\odot {\bf\tau})
= ((S_I\cup I)\cap T)+{\bf u}\odot{\bf\tau}.$$
Now, by Lemma~\ref{l4}, 
$$Q(P_1,P_2, {\textbf{s}}) = \mbox{constant, for all}\quad {\bf s} \in ((S_I\cup I)\cap T) + 
\textbf{u}\odot {\bf\tau}.$$
This completes the proof of the lemma.\end{proof}

In the proof of the next Lemma, we imitate the  arguments given in the proof 
of Lemma 2.2 and Lemma 2.3 of [HL].

\begin{lemma}\label{l2}
Let $p:\Q^2_{>0}\longto \Q$ be a rational quasi-polynomial of degree $n\geq 1$ 
with period $\tau = (\tau_1, \tau_2)\in \Q^2_{>0}$ and constant leading coefficients, {\em i.e.}, 
\begin{equation}\label{e1}
p({\bf r})=\sum_{l_1+l_2 \leq n}
p_{l_1,l_2}({\bf r})r_1^{l_1}r_2^{l_2},\quad\mbox{where}\quad{\bf r} = (r_1, r_2)\in 
\Q^2_{> 0},
\end{equation}
such that 
\begin{enumerate} \item  
$p_{l_1, l_2}({\bf r})\in \Q$ is constant for all 
$(l_1, l_2)\in\Z^2_{\geq 0}$ 
with $l_1+l_2=n$, and for every ${\bf r}\in \Q^2_{>0}$,
\item $p_{l_1, l_2}:\Q^2_{>0}\longto \Q$ are periodic functions with period 
${\bf \tau}=(\tau_1, \tau_2)$ for all $(l_1, l_2)\in\Z^2_{\geq 0}$ with 
$l_1+l_2 < n$.
\end{enumerate}

 Let $E\subseteq \R^2_{>0}$ be a 
subset of $\R^2$, such that, for every ${\bf u}\in \Z^2_{\geq 0}$, there exists 
$c_{\bf u}\in \Q$ with 
\begin{equation}\label{e2}
p({\bf r}+{\bf u}\odot{\bf \tau})=c_{\bf u}\text{ for all }{\bf r}\in E\cap\Q^2.
\end{equation}
Then for all $(l_1, l_2)\in\Z^2_{\geq 0}$ with $l_1+l_2 \leq n$, 
the coefficient function $p_{l_1, l_2}:E\cap \Q^2_{>0}\longto \Q$ is a polynomial of 
degree atmost  $n-(l_1+l_2)$.
\end{lemma}
\begin{proof}
We prove the lemma by induction on $\deg~p = n$. For $n=1$ and for ${\bf r}\in 
E\cap \Q^2$, 
we have
\begin{equation*}
c_{\bf 0} = p({\bf r})
= p_{0,0}({\bf r})+p_{(1, 0)}({\bf r})r_1+p_{(0, 1)}({\bf r})r_2.
\end{equation*}
Therefore $p_{0,0}:E\cap \Q^2\longto \Q$ is a polynomial of degree $\leq 1$ (if $p_{(1,0)} = p_{(0,1)} = 0$ then $p_{(0,0)} = $ constant). 

Now let $n\geq 2$.

Let
$$q({\bf r})=p({\bf r}+{\bf e}_2\odot{\bf\tau})-p({\bf r}) = p(r_1, r_2+\tau_2)-p(r_1, r_2).$$
Then, for ${\bf r}\in E\cap\Q^2_{>0}$, and ${\bf u}\in\Z^2_{\geq 0}$, we have
$$q({\bf r}+{\bf u}\odot{\bf \tau})=p({\bf r}+({\bf u}+e_2)\odot{\bf \tau})
-p({\bf r}+{\bf u}\odot{\bf \tau})=c_{{\bf u}+e_2}-c_{\bf u}.$$
Next we show that $q$ is a quasi-polynomial of degree $n-1$ and of period
${\bf \tau}$ with 
constant leading coefficients. Now, for every ${\bf r}\in \Q^2_{\geq 0}$

\begin{eqnarray*}
q({\bf r})&=&p(r_1, r_2+\tau_2)-p(r_1,r_2)\\
&=&\sum_{l_2\neq 0,~~l_1+l_2\leq n}p_{l_1, l_2}({\bf r})r_1^{l_1}
\left[ \tau_2^{l_2}+{{l_2}\choose{1}}\tau_2^{l_2-1}r_2+\cdots +
{{l_2}\choose{l_2-1}}\tau_2r_2^{l_2-1}\right]\\
& = & \sum_{l_1, l_2\geq 0,~~l_1+l_2\leq n-1} q_{l_1, l_2}({\bf r})r_1^{l_1}r_2^{l_2},
\end{eqnarray*}

\noindent{where} 
\begin{equation}\label{e4}q_{l_1, l_2}({\bf r}) = p_{l_1, l_2+1}({\bf r})c_{1}(l_1,l_2)+
p_{l_1, l_2+2}({\bf r})c_{2}(l_1,l_2)+\cdots +p_{l_1, n-l_1}
({\bf r})c_{n-l_1-l_2}(l_1,l_2)
\end{equation}
\noindent{and} $c_{i}(l_1,l_2)$ are positive constants.
Therefore
$q_{l_1,l_2}:\Q^2_{> 0}\longto Q$ are periodic functions with period 
$(\tau_1,\tau_2)$ and, for every ${\bf r}\in \Q^2_{>0}$, 
$$q_{l_1, n-1-l_1}({\bf r}) = p_{l_1, n-l_1}({\bf r})
c_{1}(l_1,l_2) = C_{l_1}~\mbox{ = a constant, for every}~0\leq l_1 \leq n-1.$$
Therefore, by induction hypothesis, $q_{l_1,l_2}:E\cap \Q^2_{> 0}\longto \Q$ 
 is a polynomial of degree atmost $n-1-l_1-l_2$, for all 
 $(l_1,l_2)\in \Z^2_{\geq 0}$ such that $l_1+l_2\leq n-1$.
By (\ref{e4}) and  descending induction on $(l_1,l_2)$, we deduce that
$p_{l_1, l_2+1}({\bf r}):E\cap \Q^2\longto \Q$ is a polynomial of degree 
atmost $n-(l_1+l_2+1)$.
Similarly, by considering the function 
$q'(r_1, r_2) = p({r_1+\tau_1, r_2})-
p(r_1, r_2)$ , we deduce
that $p_{l_1+1, l_2}({\bf r}):E\longto \Q$ is a polynomial of degree atmost $n-(l_1+l_2+1)$.
Now, the function $p_{0,0}:\Q^2_{\geq 0}\longto \Q$ is given by 
$$p_{0,0}({\bf r}) = c_0-\sum_{(0,0)\neq (l_1,l_2)\in \Z^2_{\geq 0}, 
l_1+l_1\leq n}p_{l_1, 
l_2}({\bf r})r_1^{l_1}r_2^{l_2}$$ 
ane hence $p_{0,0}:E\cap \Q^2_{>0}\longto \Q$ is a polynomial of degree 
atmost $n$. This completes 
the proof of the lemma\end{proof}

\begin{thm}\label{t1}Let $P_1, P_2\subset \R^d$ be rational convex polytopes such that 
$\dim{(P_1+P_2)} = d$. Let 
$Q(P_1,P_2;-):\Q^2_{\geq 0} \longto \Q$ be the rational quasi-polynomial of degree $d$,
 given by 
$$Q(P_1, P_2,;{\bf r}) = \#((r_1P_1+r_2P_2)\cap \Z^d) = 
\sum_{(l_1, l_2)\in \Z^2_{\geq 0}}p_{l_1, l_2}({\bf r})r_1^{l_1}r_2^{l_2}.$$ 
 
Then there is a finite set ${\tilde S}$ consisting of  locally closed,
 bounded subsets of 
$\R^2_{>0}$, and a finite set of polynomials
$$\{f_{l_1, l_2}^U:U\longto \Q\mid U\in {\tilde S},~~~
(l_1, l_2)\in \Z^2_{\geq 0},~~~ l_1+l_2\leq n\},$$
 such that 
\begin{enumerate}\item 
$\R^2_{>0} = \cup_{U\in {\tilde S}}\cup_{u\in \Z^2_{\geq 0}}(U+{\bf u}
\odot {\bf\tau})$ and  
$$p_{l_1,l_2}({\bf r}) = f_{l_1, l_2}^U({\bf r}), ~~~\mbox{for every}~~~ {\bf r}\in U\in {\tilde S}.$$
\item In particular,  there exist nonnegative constants
$C_{l_1,l_2}$ such that for all ${\bf r}\in \Q^2_{>0}$,
 $$|p_{l_1, l_2}({\bf r})|\leq C_{l_1, l_2},~\mbox{and}~ 
p_{l_1,d-l_1}({\bf r}) = C_{l_1, d-l_1}.$$
\end{enumerate}
\end{thm}
\begin{proof}
Let $${\tilde S} = \{(S_I\cup I)\cap T\mid S_I\in {\tilde C_P},~~I\in 
{\tilde I_P}~~(S_I\cup I)\cap T\neq \phi\}
\cup \{T_0\cap T\}.$$
Since $T_0$ is a discrete set of points, the set $T_0\cap T =$ a finite set of points.

By Lemmas~\ref{l5} and \ref{l2},  for every $U\in {\tilde S}$
there is a set of polynomials, 
$$\{f_{l_1, l_2}^U:U\longto \Q\mid ~~~
(l_1, l_2)\in \Z^2_{\geq 0},~~~ l_1+l_2\leq n\},$$
such that $p_{l_1, l_2}({\bf r}) = f_{l_1, l_2}^U({\bf r})$, for every ${\bf r}\in U$.

Now since each $p_{l_1, l_2}:\Q^2_{>0}\longto \Q$  is a periodic function of period 
$\tau = (\tau_1, \tau_2)$, we can choose 
$C_{l_1, l_2} = \max\{|f_{l_1, l_2}^U({\bf r})|\mid {\bf r}\in U,~~~U\in {\tilde S}\}$.
This proves the theorem.
\end{proof}

\begin{rmk}\label{rhl}Theorem~\ref{t1} can be generalised to the Minkowski sum of 
any finite number of polytopes, say $P_1, P_2, \ldots, P_n$ in $\R^d$ 
 where $\dim{(P_1+\cdots +P_n)} =d$:  
For this we express
$\R^d_{>0}$ as the disjoint union of locally closed 
sets, namely 
$$\R^d_{> 0} = (\bigcup_{S\in {\tilde C_n}}S)\cup 
(\bigcup_{S\in {\tilde C_{n-1}}}S)\cup\cdots (\bigcup_{S\in 
{\tilde C_1}}S)\cup S_0,$$
where ${\tilde C_k}$ denotes the set of $k$-cells obtained by the hyperplane 
arrangements as given in (\ref{ehyp}). 
Now, given any $I\in {\tilde C_k}$ there exists a unique 
cell $S_I\in {\tilde C_n}$ such that 
$I\subseteq {\bar S_I}$, the closure of $S_I$ in $\R^d_{>0}$, and
$$S_I\subseteq H_{j_1}({\bf z}_1)^{+o}\cap H_{j_2}({\bf z}_2)^{+o}\cap 
\cdots\cap H_{j_{n-k}}({\bf z}_{n-k})^{+o},$$
where
$I\subseteq  H_{j_1}({\bf z}_1)\cap H_{j_2}({\bf z}_2)\cap 
\cdots\cap H_{j_{n-k}}({\bf z}_{n-k}).$
Now one can check that 
 $Q(P_1, \ldots, P_n, {\bf s}) =$ constant for all ${\bf s}\in S_I\cup I$
and hence the proof follows imitating the rest of the arguments.
\end{rmk}

\begin{thm}\label{tcb}
Let $P$ be a  convex rational polytope in $\R^d = \R^{d-1}\times \R$ of dimension $d$ 
and let 
$$i(P_\lambda, n) = \sum_{i=0}^{\dim(P_\lambda)} 
C_i(P_\lambda,n)n^i,\ \ \text{for}\ \ n\in\Z_{\geq 0},$$
be the  Ehrhart quasi-polynomial for $P_\lambda := P\cap \{z=\lambda\} 
\subset \R^{d-1}\times \{\lambda\}$, for $\lambda \in \Q_{\geq 0}$. Then 
there are  constants, independent of $n$,  ${\tilde c_i}(P)$ and   
${\tilde c_{d-1}}(P_\lambda)$ such that 
for $0\leq i \leq d-1$, 
$$|C_i(P_\lambda, n)|\leq {\tilde c_i}(P),\quad
\mbox{and}\quad C_{d-1}(P_\lambda, n)= {\tilde c_{d-1}}(P_\lambda)
\quad\mbox{provided}~~~\lambda n \in \Z_{\geq 0}.$$  
\end{thm}
\begin{proof}We say a polytope $P$  satisfies $(\star)$ condition if 
all the vertices of $P$ lie in the union of two hyperplanes $\{z=a_1\}\cup\{z=a_2\}$,
 for some rational numbers $a_1< a_2$.

\vspace{5pt}

First we prove the theorem for
 $P$ with the additional $(\star)$ condition. 
 
Let $P_1 = P\cap \{z = a_1\}$ and $P_2 = P\cap \{z = a_2\}$.
Then 
$$P_1 = \mbox{convex hull}\{({\bf v}_1,a_1), \ldots,  ({\bf v}_m,a_1)\}~~~
\mbox{and}~~~P_2 = \mbox{convex hull}\{({\bf w}_1,a_2), 
\ldots,  ({\bf w}_n,a_2)\}, $$
for a set of some rational points $\{{\bf v}_i, {\bf w}_j\}_{i,j}\subset 
\R^{d-1}$.

Note that, for $\lambda \in [a_1, a_2]$, there is a decomposition as a 
Minkowski sum
$$P\cap \{z =\lambda\} = r_{\lambda}P_1+r'_{\lambda}P_2,\quad 
\mbox{ where}\quad
r_{\lambda}=\frac{a_2-\lambda}{a_2-a_1},\quad 
r'_{\lambda}=\frac{\lambda-a_1}{a_2-a_1}.$$
Let 
${\tilde P_1} = \mbox{convex hull of}~\{{\bf v}_1, \ldots,  {\bf v}_m\}~~~
\mbox{and}~~~{\tilde P_2} = \mbox{convex hull}~\{{\bf w}_1, 
\ldots,  {\bf w}_n\}.$
Note that $\dim{({\tilde P_1}+{\tilde P_2})} = d-1$, as there exists an 
isomtric map
${\tilde P_1}+{\tilde P_2} \longto P_1+P_2$ given by 
$$\sum_i\lambda_i{\bf v}_i+ \sum_j\mu_j{\bf w}_j\mapsto 
(\sum_i\lambda_i{\bf v}_i+ \sum_j\mu_j{\bf w}_j, a_1+a_2) = 
(\sum_i\lambda_i{\bf v}_i, a_1)+ (\sum_j\mu_j{\bf w}_j, a_2)$$
and therefore 
$\dim{(P_1+P_2)} = \dim{(2(P\cap\{z={a_1+a_2}/{2}\}))} = d-1$.
Hence, by Theorem~\ref{t1}, we have 
\begin{equation}\label{ec}Q({\tilde P_1}, {\tilde P_2},{\bf r}) := 
|(r_1{\tilde P_1}+r_2{\tilde P_2})\cap \Z^{d-1}| =
\sum_{l_1+l_2\leq d-1}p_{l_1, l_2}({\bf r})r_1^{l_1}r_2^{l_2},\quad\mbox{for}\quad
{\bf r}\in \Q^2_{\geq 0},\end{equation} 
where,  there exist  constants $C_{l_1, l_2}$ and $C_{l_1}$ such that
$$|p_{l_1, l_2}({\bf r})| \leq C_{l_1, l_2}\quad{and}\quad 
 p_{l_1,d-1-l_1}({\bf r}) = C_{l_1}\quad\mbox{for all}\quad {\bf r}\in\Q^2_{>0}.$$
  
If $\lambda n\in \Z_{> 0}$, then one can check that 
$$i(P\cap \{z = \lambda\}, n) = \#((r_{\lambda}nP_1+r'_{\lambda}nP_2)
\cap \Z^{d}) = 
\#((r_{\lambda}n{\tilde P_1}+r'_{\lambda}n{\tilde P_2})\cap \Z^{d-1}).$$
Therefore, by (\ref{ec}),
$$i(P\cap \{z = \lambda\}, n) = \sum_{l_1+l_2\leq d-1}p_{l_1, l_2}
(r_{\lambda}n, {r'_\lambda} n)r_{\lambda}^{l_1}
{r'_\lambda}^{l_2}n^{l_1+l_2} = 
\sum_{i=0}^{d-1}{C_i}(P_\lambda, n)n^i,$$
where 
$${C_i}(P_\lambda, n) = \sum_{l_1+l_2=i}p_{l_1, l_2}
(r_{\lambda}n, {r'}_{\lambda}n)r_{\lambda}^{l_1}{r'_\lambda}^{l_2}.$$
 Now $ a_1 < \lambda < a_2$ and $n\in \Z_{>0}$ implies  $(r_{\lambda}n, 
{r'_\lambda}n)
\in \Q^2_{>0}$ and therefore, by Theroem~\ref{t1} and Theorem~\ref{tl}, 
there exist constants $c_i'(P)$ and 
$c'_{d-1}(P_\lambda)$ such that 
$|{C_i}(P_\lambda, n)|\leq c'_i(P)$ and 
${C_{d-1}}(P_\lambda, n) = c'_{d-1}(P_\lambda) = \mbox{rVol}_{d-1}(P_\lambda)$
(see Definition~\ref{rv} for a discussion of the relative volume 
$\mbox{rVol}_{d-1}$).

Now ${\tilde c_i}(P) := 
\max\{c'_i(P), C_i(P_{a_1}, n), C_i(P_{a_2}, n)\}$,  
is finite, by the 
theory of Ehrhart polynomials for the rational polytopes 
$P_{a_1}$ and $P_{a_2}$. Moreover, $C_{d-1}(P_{a_i}, n)$ is constant 
($= 0$, if $\dim(P_{a_i}) <d-1$).

This proves the theorem for a rational polytope $P$ which satisfies the  
condition $(\star)$.

Consider the projection map $\pi:\R^d\longto \R$ to the last coordinate.
Let $b_1<b_2<\cdots <b_l$, where $b_i$ are 
the images of the vertices of the polytope $P$. 
Now $P_{b_m}^{b_{m+1}}:=P\cap \{b_m\leq z \leq b_{m+1}\}$ satisfy the 
condition $(\star)$. 
Hence the proof of the theorem follows by taking 
$${\tilde c_i}(P) = \max\{{\tilde c_i}(P_{b_m}^{b_{m+1}})\mid 
1\leq m\leq l-1\}\quad\mbox{and}\quad 
{\tilde c_{d-1}}(P_\lambda) =  {\tilde c_{d-1}}((P_{b_m}^{b_{m+1}})_\lambda),~~\mbox{ if}~~\lambda \in 
[b_m, b_{m+1}].$$
\end{proof}

\section{Main theorem}

We now resume the study of Eto's set $\sP_D$, and the decompositions 
$\sP_D = \cup_{j=1}^sP'_j$ and ${\bar \sP_D} = \cup_{j=1}^sP_j$, as 
discussed in Section~3. We will make use of properties of 
relative volumes, recalled below in the  appendix (see Lemmas~\ref{rvol} and \ref{la}).
 
\begin{notations}\label{d4}\begin{enumerate}
\item  $F(Q) = \{\mbox{the facets of}~Q\}$ and $v(Q) = \{\mbox{the 
vertices of}~Q\}$, where 
$Q$ is a convex polytope.
\item 
Let $v(\sP_D) := \cup_{j=1}^s\pi(v(P_j))$, where 
 $\pi:\R^d \longto \R$ is the projection 
map given by projecting  to the last coordinate $z$ and the set 
$\pi(v(P_j)) = \{\rho_{j_1}, \ldots, \rho_{j_m}\},\quad\mbox{with}~
\rho_{j_1}<\rho_{j_2}<\cdots < \rho_{j_{m_j}}$. 
\item Let $S = \{m/q\mid q = p^n, ~~m, n\in \Z_{\geq 0}\}\setminus v(\sP_D)$.
\end{enumerate}
\end{notations}

\begin{lemma}\label{lr1}Let $P_j$ be a convex polytope as given in 
Notations~\ref{n3}~(8). Let 
$$i(P_j, n, \lambda) = i(P_{j\lambda}, n)) = \sum_{i=0}^{\dim(P_{j\lambda})} 
C_i(P_{j\lambda},n)n^i~~~\mbox{for}~~ n\in\Z_{\geq 0}$$
be the  Ehrhart quasi-polynomial for 
$P_{j\lambda} = P_j\cap \{z=\lambda\}$, for $\lambda\in \R_{\geq 0}$.
Then 
for $\lambda \in S$ and $q\lambda \in \Z$,
we have
\begin{enumerate}
\item  
$\sum_{j=1}^sC_{d-1}(P_{j\lambda}, q) = 
f_{R, {\bf m}}(\lambda)$.
\item $C_{d-2}(P_{j\lambda}, q) = \frac{1}{2}\sum_{{\tilde F}\in F(P_j)}
rVol_{d-2}({\tilde F}\cap\{z=\lambda\})$.
\item For every $i\leq d-3$, we have $|C_i(P_{j\lambda}, q)| \leq 
{\tilde c_i}(P_{j})$ for some 
constants ${\tilde c_i}(P_{j})$ independent of $\lambda \in S$.
\end{enumerate}
\end{lemma}
\begin{proof} Let $\lambda = m_0/q_0$ then $q\lambda \in \Z$ implies 
$q\geq q_0$. 

\vspace{5pt}

\noindent{(1)}  
By Theorem~{1.1} of [MT], for any $\lambda \in \R$, we have 
$f_{R, {\bf m}}(\lambda) = \mbox{Vol}_{d-1}(\sP_D\cap\{z=\lambda\})$.
Also,  by the proof of 
Theorem~1.1 of [MT] (see the proof of the claim there), we have 
$$C_{d-1}(P_{j\lambda},q) = C_{d-1}(P_{j\lambda})
 =  rVol_{d-1}(P_{j\lambda})~~\mbox{and}~~ 
 \sum_jrVol_{d-1}(P_{j\lambda}) = \mbox{Vol}_{d-1}(\sP_D\cap\{z=\lambda\}).$$

\noindent{(2)}~Case~(a):~~If $\lambda \in (\rho_{j_1}, \rho_{j_{m_j}})$, where 
$\{\rho_{j_1}, \ldots, \rho_{j_{m_j}}\} = \pi(v(P_j))$ (as in Notations~\ref{d4})
then 
$\dim{P_{j\lambda}} = d-1$. 
Note that the set of facets of $P_{j\lambda}$
$$ F(P_{j\lambda}) = \{{\tilde F}\cap\{z=\lambda\}
\mid  {\tilde F}\in F(P_{j}),~~\dim{\tilde F}\cap
\{z= \lambda\}=d-2\}.$$ 
By Lemma~\ref{l8}~(3), for any $F \in F(P_{j\lambda})$, we have
 $A(q_0F\cap\{z=m_0\})\cap \Z^d \neq \phi$.
Hence, for all $q\geq q_0$, Theorem~\ref{ehrhart} implies 
$$C_{d-2}(q_0P_{j\lambda}, \frac{q}{q_0}) = 
\sum_{F\in F(P_{j\lambda})}\frac{rVol_{d-2}(q_0F_{\lambda})}{2} = 
\frac{q_0^{d-2}}{2}\sum_{F\in F(P_{j\lambda})}rVol_{d-2}(F_{\lambda}).$$
Moreover $rVol_{d-2}({\tilde F}\cap\{z=\lambda\}) = 0$ if 
$\dim({\tilde F}\cap\{z=\lambda\}) <d-2$. 

\noindent{(2)}~Case~(b):~~For $\lambda \not\in [\rho_{j_1}, \rho_{j_{m_j}}]$ we  know
$i(P_{j\lambda},q)=0$, for all $q$.

\noindent{(3)} follows by Theorem~\ref{tcb}.
This proves the lemma.\end{proof}

\begin{lemma}\label{lr2} Let $Q_{\alpha} = P_i\cap P_j$ or
$E_{j\nu}$ (as in Notations~\ref{n3} and Remark~\ref{r1}), where  $P_i \neq P_j$. Let 
$i(Q_{\alpha\lambda}, n)) = \sum_{i=0}^{d-2} 
C_i(Q_{\alpha\lambda},n)n^i$ be the
  Ehrhart quasi-polynomial
for
$Q_{\alpha\lambda} = Q_\alpha \cap \{z= \lambda\}$ where $\lambda 
\in \Q_{\geq 0}$. 
Then, for $\lambda \in S$ and $q\lambda \in \Z_{\geq 0}$, we have 
\begin{enumerate}
\item $C_{d-2}(Q_{\alpha\lambda}, q) =  
rVol_{d-2}(Q_{\alpha\lambda})$. Moreover 
\item \begin{enumerate}
\item if $\dim{Q_\alpha} = d-1$ then $|C_i(Q_{\alpha\lambda}, q)|
\leq {\tilde c_i}(Q_\alpha)$, for every $i\leq d-3$, 
 and 
\item if $\dim{Q_\alpha} \leq d-2$ then 
$i(Q_{\alpha \lambda}, q) \leq C_\alpha q^{d-3}$.
\end{enumerate}\end{enumerate}
\end{lemma}
\begin{proof}Let $\lambda = m_0/q_0$ and let $q\geq q_0$.

\noindent{(1)}~
Case~(a):~~If  $\dim{Q_{\alpha\lambda}} = d-2$ then 
by Lemma~\ref{l8}~(3), for $q\lambda\in \Z_{\geq 0}$,
 $A(q_0Q_{\alpha\lambda})\cap \Z^d = A(Q_{\alpha})\cap\Z^d \neq \varphi$. Hence, by Theorem~\ref{ehrhart}
\begin{equation}\label{evol}C_{d-2}(q_0Q_{\alpha\lambda}, q/q_0) =  
C_{d-2}(q_0Q_{\alpha\lambda}) =
rVol_{d-2}(q_0Q_{\alpha\lambda}) = 
q_0^{d-2}rVol_{d-2}(Q_{\alpha\lambda}).\end{equation}

\noindent{(1)}~Case~(b):~~Let $\dim{Q_{\alpha\lambda}} < d-2$.
By Lemma~\ref{la}, for a convex rational polytope $Q_\alpha$, there exists a 
constant $C_\alpha$ such that 
$i(Q_{\alpha \lambda}, q) \leq C_\alpha q^{\dim{Q_{\alpha\lambda}}}$.
Therefore $C_{d-2}(Q_{\alpha\lambda}, q) = 0 = 
rVol_{d-2}(Q_{\alpha\lambda})$.
This proves the first assertion.

\vspace{5pt}

\noindent{(2)}\quad if $\dim{Q_\alpha} = d-1$ then 
by Lemma~\ref{rvol}, there exists a map $\varphi_\alpha:\R^d\longto \R^{d-1}$ and 
$z_\alpha \in \Z^{d-1}$ such that  
$i(Q_{\alpha\lambda},q) = i(\varphi_{\alpha}(Q_\alpha)_\lambda, q)$,
 where 
$\varphi_\alpha(Q_\alpha)$ is a $d-1$-dimensional  polytope in $\R^{d-1}$
Hence Assertion~2~(a) follows from Theorem~\ref{tcb}.
If $\dim{Q_\alpha} \leq  d-2$ then $Q_\alpha = P_i\cap P_j$, as 
$\dim{E_{j\nu}} = d-1$ if $E_{j\nu}\neq \varphi$.  Without loss of 
generality we can assume that $\dim{Q_\alpha} \geq 1$. Since 
$Q_\alpha$ is tranversal to the hyperplane $\{z = 0\}$, 
we have $\dim{Q_{\alpha\lambda}} < \dim{Q_\alpha} $, 
for all $\lambda $. Hence, by Lemma~\ref{la}, we have Assertion~(2)~(b).
This completes the proof of the lemma.\end{proof}

\begin{lemma}\label{lr3}For $Q'_\alpha$, where $\alpha \in I_1$, we have
 $i(Q'_{\alpha\lambda}, q) 
= {\tilde c'}_{\alpha}(\lambda)q^{d-3}$,
where $|{\tilde c'}_{\alpha}(\lambda)|< {\tilde c_\alpha}$, 
for all $\lambda\in\Q_{\geq 0}$.
\end{lemma}
\begin{proof}Note that $\dim{Q'_{\alpha}} 
\leq d-2$, for $\alpha \in I_1$. Hence, by Lemma~\ref{l8}~(1), if 
$\dim{Q'_{\alpha\lambda}} = d-2$, for some $\lambda \in \R_{\geq 0}$ then 
$\lambda \in v(\sP_D)$. Therefore $\lambda \in S$ implies  
$\dim{Q'_{\alpha\lambda}} \leq d-3$.  
 Hence the proof follows by Lemma~\ref{la}.\end{proof}

\begin{defn}\label{d5}For a pair $(R,{\bf m})$, where $R$ is a standard 
graded ring  of dimension $d$, we define

\noindent{(1)}\quad a sequence of functions 
$g_n:[0, \infty)\longto \R$, given by 
$$g_n(\lambda)= g_n(\lfloor \lambda q\rfloor/q) = \frac{1}{q^{d-2}}
\left(\ell(R/\textbf{m}^{[q]})_{\lfloor \lambda q\rfloor}
- {\bar f_n}(\lambda)q^{d-1}\right),$$
where $f_{R, {\bf m}}$ is the \mbox{HK} density function for $(R,{\bf m})$
 (see Theorem~\ref{hkd}) and ${\bar f_n}(\lambda):= 
f_{R, {\bf m}}(\frac{{\lfloor \lambda q\rfloor}}{q})$.

\vspace{5pt}

\noindent{(2)}\quad
We also define the {\it $\beta$ density function} $g_{R, {\bf m}}:[0, 
\infty)\longto \R$, given by 
$$g_{R, {\bf m}}(\lambda) = \sum_{\{F\in F(P_j), 
F\subseteq \partial(C_D)\}_{j}}\frac{
\text{\mbox{rVol}}_{d-2}\big(F_\lambda\big)}{2}
- \sum_{\{F\in F(P_j), 
F\subseteq F'\in F(C_u), u\in L(P_D)\}_{j}}  \frac{
\text{\mbox{rVol}}_{d-2}\big(F_\lambda\big)}{2},$$
where $P_j$ and $v(\sP_D)$ are as in Notations~\ref{n3} and in 
 Notations~\ref{d4} and $F_\lambda = F\cap \{z = \lambda\}$.
Hence, by Lemma~\ref{l7}~(2), 
\begin{equation}\label{d*}g_{R, {\bf m}}(\lambda) = 
\mbox{rVol}_{d-2}\left(\partial({\sP_D})\cap \partial(C_D)\cap\{z=\lambda\}\right)
- \frac{\mbox{rVol}_{d-2}\left(\partial({\sP_D})\cap \{z=\lambda\}\right)}{2}.
\end{equation}
\end{defn}

\begin{rmk}\label{r4}By Lemma~\ref{rvol}, for a facet $F$ of $P_j$,
 there exists an
invertible affine transformation
$\varphi_{F}:A(F)\longto \R^{d-1}$ and $z_F\in \varphi(A(F))\cap \Z^{d-1}$
such that for $\lambda \in S = \{m/p^n\mid m, n\in \Z_{\geq 0}\}$, 
$$\mbox{rVol}_{d-2}(F\cap \{z = \lambda\}) = 
\mbox{Vol}_{d-2}(\varphi_{F}(F)\cap \{z_F = \lambda \}).$$ 
Note that 
$\cup_{j}\{v(F)\mid F\in F(P_j)\}\subseteq v(\sP_D)$.
Therefore, by Theorem~2.3 of [MT], 
the function $\psi_{F}:[0, \infty)\setminus v(\sP_D)\longto [0,\infty)$
 given by 
$\lambda \to \mbox{Vol}_{d-2}(\varphi_{F}(F)\cap \{z_{F} = \lambda \})$ 
is continuous.

Thus the  function $g_{R, {\bf m}}$ is  a compactly supported function and is 
continuous outside the finite set $v(\sP_D)$. 
Moreover, by Lemma~3.4 of [MT], 
$g_{R, {\bf m}}$ is a piecewise polynomial function.
\end{rmk}

\begin{lemma}\label{ml1}Let  $(R, {\bf m})$ denote the ring
associated to 
 a given toric pair $(X, \Delta)$ (as in Notations~\ref{n1}) of 
dimension $d-1\geq 1$. If $\lambda \in \R_{\geq 0}$ and $q= p^n\in\N$ are  
such that $\lambda_n := {\lfloor \lambda q\rfloor}/q \in S$ then  
 there exists a constant ${\tilde C_{\sP_D}}$ such that 
$$ g_n(\lambda) = g_{R, {\bf m}}(\lambda_n)
 + {\tilde c}(\lambda_n)/q,~~\mbox{where}~~ 
|{\tilde c}(\lambda_n)|\leq {\tilde C_{\sP_D}}.$$
\end{lemma}
\begin{proof} For a polytope $P$, we 
have $i(P, q,  \lambda_n q) = i(P_{\lambda_n}, q)$ and  
$\ell(R/\textbf{m}^{[q]})_{\lambda_n q}
= i((\mathcal{P}_D)_{\lambda_n}, q)$, where, by Remark~\ref{r1}~(\ref{***}), we have
$$i((\mathcal{P}_D)_{\lambda_n}, q) = 
\sum_ji((P_j)_{\lambda_n}, q)-\sum_{j< k}i((P_j\cap P_k)_{\lambda_n}, q)
- \sum_{j,\gamma}i((E_{j\gamma})_{\lambda_n}, q) +
\sum_{\alpha\in I_1}\epsilon_\alpha i((Q'_{\alpha})_{\lambda_n}, q).$$

Now, by Lemma~\ref{lr1}, for $1\leq j\leq s$, we have 
 $$i((P_j)_{\lambda_n}, q)  
- \mbox{rVol}_{d-1}((P_j)_{\lambda_n})q^{d-1} =  \frac{1}{2}\sum_{\{F\in F(P_j)\}}
\mbox{rVol}_{d-2}{F_{\lambda_n}}q^{d-2}+ {\tilde c_j}(\lambda_n)q^{d-3}$$
$$= 
\left[\sum_{\{F\mid F=P_i\cap P_j\}_i}
\frac{\mbox{rVol}_{d-2}{F_{\lambda_n}}}{2} + 
\sum_{\{F\mid F\subseteq \partial(C_D)\}}
\frac{\mbox{rVol}_{d-2}F_{\lambda_n}}{2} + \sum_{\{\nu\}}
\frac{\mbox{rVol}_{d-2}(E_{j\nu})_{\lambda_n}}{2}\right]q^{d-2}+
{\tilde c_j}(\lambda_n)q^{d-3},$$
where $F\in F(P_j)$ and  $|{\tilde c_j}(\lambda_n)|\leq {\tilde c_j}$, 
for some constant $c_j$ independent of $\lambda_n$.

Hence, by Lemma~\ref{lr1}, 
$$\sum_ji((P_j)_{\lambda_n}, q) - 
f_{R,{\bf m}}(\lambda_n)q^{d-1} =$$
$$\left[ \sum_{\{F= P_i\cap P_j\mid {i<j}\}_j}
\mbox{rVol}_{d-2}{F_{\lambda_n}} + 
\sum_{\{F\mid F\subseteq \partial(C_D)\}_j}
\frac{\mbox{rVol}_{d-2}{F_{\lambda_n}}}{2} + \sum_{\{j, \nu\}}
\frac{\mbox{rVol}_{d-2}(E_{j\nu})_{\lambda_n}}{2}\right]q^{d-2}+ 
{\tilde c_j}(\lambda_n)q^{d-3}$$
and by Lemma~\ref{lr2},
$$\sum_{i<j}i((P_i\cap P_j)_{\lambda_n}, q) =  
\sum_{\{F= P_i\cap P_j\}_{i<j}}
\mbox{rVol}_{d-2}{F_{\lambda_n}} q^{d-2} + {\tilde c_{ij}}(\lambda_n)q^{d-3},~~
\mbox{where}~~|{\tilde c_{ij}}(\lambda_n)|\leq {\tilde c_{ij}}.$$

Now, by Lemmas~\ref{lr2} and \ref{lr3}, we have 
$$\frac{1}{q^{d-2}}\ell(\frac{R}{{\bf m}^{[q]}})_{\lambda_n q}-
f_{R, {\bf m}}(\lambda_n)q = \sum_{\{F\in F(P_j),~ F\subseteq \partial(C_D)\}_j}
\frac{\mbox{rVol}_{d-2}{F_{\lambda_n}}}{2} - \sum_{\{j, \nu\}}
\frac{\mbox{rVol}_{d-2}(E_{j\nu})_{\lambda_n}}{2}+
\frac{{\tilde c}(\lambda_n)}{q},$$
where $|{\tilde c}(\lambda_n)|\leq {\tilde C_{\sP_D}}$.
 Hence 
$g_n(\lambda) = g_{R, {\bf m}}(\lambda_n)
 + {\tilde c}(\lambda_n)/q$, where $|{\tilde c}(\lambda_n)|\leq 
{\tilde C_{\sP_D}}$.
This implies the lemma. 
\end{proof}

\begin{rmk}By construction it follows that 
$\mbox{Support}(g_{R, {\bf m}})\cup_n\mbox{Support}(g_n)
 \subseteq \pi({\bar{\sP_D)}}$, which is a compact set and 
where $\pi:\R^d\to \R$ is the projection map as in Notations~\ref{d4}~(2).
\end{rmk}

\begin{lemma}\label{l42}The function $g_{R, {\bf m}}$ is a 
compactly supported funtion and is continuous on 
$[0, \infty)\setminus v(\sP_D)$.
 \begin{enumerate}
\item For any given compact set 
$V\subseteq [0, \infty)\setminus v(\sP_D)$, 
the sequence $g_n|_V$ converges uniformly to  $g_{R, {\bf m}}|_V$.
\item $\int_0^\infty g_n(\lambda)d\lambda = \int_0^\infty 
g_{R, {\bf m}}(\lambda)d\lambda +O(1/q)$.\end{enumerate}
\end{lemma}
\begin{proof} By Remark~\ref{r4}, the function $g_{R, {\bf m}}$ is a 
compactly supported funtion and is continuous outside $v(\sP_D)$.

\vspace{5pt}
\noindent(1)\quad Let us fix  
 a compact set $V\subseteq [0, \infty)\setminus 
v(\sP_D)$. 
We can assume $V\subseteq 
[0, \pi({\bar \sP_D})]\setminus v(P_D)$.  Now, let  
$[0, \pi({\bar\sP_D})]\setminus v(\sP_D) = 
\cup_{i=1}^m (b_i,b_{i+1})$, where 
$f_{R,{\bf m}}|_{(b_i, b_{i+1})} = P_i$, for some polynomial function $P_i$.
We choose $q_0 = p^{n_0}$ such that 
$$V\subseteq \cup_{i=1}^m[b_i+2/q_0, b_{i+1}-2/q_0] \subseteq
\cup_{i=1}^m[b_i, b_{i+1}].$$ 
Now,  $q\geq q_0$ and $\lambda \in V\cap (b_i, b_{i+1})$ imply  $\lambda_n\in 
[b_i+1/q_0, b_{i+1}-1/q_0]$. Hence 
$$|g_{R, {\bf m}}(\lambda)-g_{R, {\bf m}}(\lambda_n)| = 
|P_i(\lambda)-P_i(\lambda_n)| \leq C_i/q,$$ 
where $C_i$ is a constant determined by $P_i$ and is independent of $q$.
By Lemma~\ref{ml1}, we have $|g_n(\lambda)-g_{R, {\bf m}}(\lambda)|\leq 
C_i/q + {\tilde C_{\sP_D}}/q$, for all $\lambda \in V\cap (b_i, b_{i+1})$ and 
$q\geq q_0$. Hence $|g_n(\lambda)-g_{R, {\bf m}}(\lambda)|\leq c_0/q$,
for alll $\lambda \in V$ and $q\geq q_0$.
This proves part~(1) of the lemma.

\vspace{5pt}
\noindent~(2)\quad Fix  $q_0 = p^{n_0}$ such that 
$\cup_{i=1}^m[b_i+2/q_0, b_{i+1}-2/q_0] \subseteq
\cup_{i=1}^m[b_i, b_{i+1}]$. 
For $q=p^n$, if $q\geq q_0$ then 
$V_1:=\cup_{i=1}^m[b_i+2/q, b_{i+1}-2/q] \subseteq
\cup_{i=1}^m[b_i, b_{i+1}]$. 
Now, arguing as above one can deduce that
there is a constant $c_0$ such that
$|g_n(\lambda)-g_{R, {\bf m}}(\lambda)|\leq c_0/q$,
for alll $\lambda \in V_1$.
Moreover $\mu([0, \pi({\bar \sP_D})\setminus V_1)\leq 4m/q$. Note that all 
the functions $g_n$ and $g$ are
bounded with support in $[0, \pi{\bar \sP_D}]$. Hence 
 $$\int_0^\infty |g_n(\lambda) - g(\lambda)|
d\lambda\leq \int_{V_1} |g_n(\lambda) -g(\lambda)|d\lambda +
\int_{[0, \infty)\setminus V_1}|g_n(\lambda) -
g(\lambda)|d\lambda = O(1/q).$$

The same assertion follows for $q<q_0$, by the boundedness of $g$ and $g_n$.
This proves the part~(2) of the lemma and hence the lemma.
 \end{proof}

\begin{lemma}\label{ml} Let $f:[0, \infty)\to [0, \infty)$ be a 
continuous compactly 
supported piecewise polynomial function. For $q= p^n$, let
 ${\tilde f}_n:[0, \infty)\longto [0, \infty)$ be the function 
given by 
${\tilde f}_n(x) = f({\lfloor qx\rfloor}/{q})$.
Then, for all $q=p^n$, we have 
$$\int_0^\infty\tilde{f}_n(x)\ dx= 
\int_0^\infty f(x)\ dx\ +\ O(1/q^2).$$
\end{lemma}
\begin{proof} We assume the following claim for the moment.

\vspace{5pt}
\noindent{\bf Claim}\quad If $P(x)\in \R[x]$ is a polynomial 
function and  ${\tilde P}_n(x):= P({\lfloor qx\rfloor}/{q})$ then 
$$\int_0^A[P(x)-{\tilde P_n}(x)]dx = P(A)/2q+ O(1/q^2).$$

Now, since $f$ is a compactly supported piecewise polynomial continuous 
function, 
 there exist 
$0 = b_0 < b_1< \ldots< b_{\nu+1}$ and polynomials $P_0(x), \ldots, 
P_{\nu}(x) 
\in \R[x]$ such that 
  $f\mid_{[b_i,b_{i+1}]} = P_{i}(x)$ 
and $\mbox{Support}(f) \subseteq [b_0, b_{\nu+1}]$ and $f(b_0) = f(b_{\nu+1}) =0$.

Now $$\int_0^\infty f(x)dx -  
\int_0^\infty \tilde{f}_n(x)dx = \int_0^\infty \left(f(x) -  
\tilde{f}_n(x)\right)dx = \sum_{i=0}^\nu\int_{b_i}^{b_{i+1}}
(P_i(x)-({\tilde P_i})_n(x))dx $$
$$ = \sum_{i=0}^\nu\left[\int_{0}^{b_{i+1}}
(P_i(x)-({\tilde P_i})_n(x))dx -\int_{0}^{b_{i}}
(P_i(x)-({\tilde P_i})_n(x))dx\right], $$ 
where $({\tilde P_i})_n(x) = P_i({\lfloor qx\rfloor}/{q})$.
Therefore, by the above claim
$$\int_0^\infty [f(x) - \tilde{f}_n(x)]dx = 
\sum_{i=0}^\nu\left[\frac{P_i(b_{i+1})}{2q}- \frac{P_i(b_{i})}{2q}
\right]+O(\frac{1}{q^2})=
\sum_{i=0}^\nu\left[\frac{f(b_{i+1})}{2q}- \frac{f(b_{i})}{2q}
\right]+O(\frac{1}{q^2}) = O(\frac{1}{q^2}).$$

\vspace{5pt}

\noindent{\underline{Proof of the claim}}:\quad We can assume without loss of 
generality that $P(x) = x^l$, for some $l\geq 1$. Note that there is 
$q_0 = p^{n_0}$ such that, for any $q\geq q_0$, there is $l_0$ such that
$l_0/q\leq A < (l_0+1)/q$. Then
$$\int_0^A(P(x)-{\tilde P}_n(x))dx  = 
\int_0^{l_0/q}(P(x)-{\tilde P}_n(x))dx + O(1/q^2),$$
 as  the inequalities
$(A-l_0/q) \leq 1/q$ and $l\geq 1$ implies 
$|\int_{l_0/q}^A(P(x)-{\tilde P}_n(x))dx| = O(1/q^2)$.
It is also obvious that 
$P(A)/2q+ O(1/q^2) = P(l_0/q)/2q+ O(1/q^2)$.

Now, if $m\in \Z_{\geq 0}$ then 
$$\int_{\frac{m}{q}}^{\frac{m+1}{q}}(P(x)-{\tilde P}_n(x))dx = 
\int_{\frac{m}{q}}^{\frac{m+1}{q}}(x^l- (m/q)^l)dx = \int_{
\frac{m}{q}}^{\frac{m+1}{q}}
\left[\sum_{i=0}^{l-1}{ l\choose i }(m/q)^{i}(x-m/q)^{l-i}\right] dx$$
$$= \int_{0}^{1/q}
\left[\sum_{i=0}^{l-1}{ l\choose i }(m/q)^{i}(t)^{l-i}\right] dt 
= \sum_{i=0}^{l-1}{ l\choose i }(m/q)^{i}\frac{l-i+1}{q^{l-i+1}}.$$
Now $\sum_{m=0}^{qA}m^i/q^{l+1} = A^{i+1}/q^{l-i}(i+1)$. 
This implies 
$$\int_0^A(P(x)-{\tilde P}_n(x))dx = 
\sum_{i=0}^{qA}\int_{m/q}^{m+1/q}(P(x)-{\tilde P}_n(x))dx = \sum_{i=0}^{l-1}
{l\choose i}\frac{1}{(l-i+1)(i+1)}\frac{A^{i+1}}{q^{l-i}}$$
$$=\frac{A^l}{2q}+\cdots +\frac{A^2}{2q^{l-1}}+ \frac{A}{(l+1)q^{l}} 
= P(A)/2q + O(1/q^2).$$
This proves the claim and hence the lemma.
\end{proof}

\noindent{\underline{Proof of the Main Theorem}}: Now the proof follows by 
putting together Definition~\ref{d5}~(\ref{d*}), Lemma~\ref{l42} and 
taking $f = f_{R, {\bf m}}$ in Lemma~\ref{ml}. $\hfill \Box$

\vspace{5pt}

\noindent{\underline{Proof of Corollary}~\ref{thebeta}:\quad
Let $g_{R, {\bf m}}$ and $g_n$ denote the function as given in 
 Definition~\ref{d5}. 

 The formula for the integral of $g_{R, {\bf m}}$ follows from
the definition of $g_{R, {\bf m}}$ and Fubini's theorem.

For $q=p^n$ and 
for the function ${\tilde f_n}$ given by 
${\tilde f_n}(x) = f_{R, {\bf m}}({\lfloor qx\rfloor}/{q})$, we have 
\begin{eqnarray*}
\int_0^\infty g_n(\lambda)\ d(\lambda)
&=&\frac{1}{q^{d-2}}
\int_0^\infty\left(\ell\big((R/\textbf{m}^{[q]})_{\lfloor q\lambda
\rfloor}\big)-{\tilde f_n}(\lambda)q^{d-1}\right)d\lambda\\
&=& \frac{1}{q^{d-2}}
\left(\sum_{m=0}^{\infty}\int_{\frac{m}{q}}^{\frac{m+1}{q}}
\ell\big((R/\textbf{m}^{[q]})_{\lfloor q\lambda\rfloor}\big)
-q^{d-1}\int\tilde{f}_n(\lambda)d\lambda \right)\\
&=& \frac{1}{q^{d-2}}
\left(\sum_{m=0}^{\infty}\frac{1}{q}\ell\big((R/\textbf{m}^{[q]})_{m}\big)
-q^{d-1}\int\tilde{f}_n(\lambda)d\lambda\right).
\end{eqnarray*}
By Lemma~\ref{ml},
$$\int_0^\infty g_n(\lambda)\ d(\lambda) = 
\frac{\ell(R/{\bf m}^{[q]})}{q^{d-1}}
-q\int_0^\infty f_{R, {\bf m}}(\lambda)d\lambda + O(1/q).$$
hence by Theorem~1.1 of [T2] and
part~(2) of Lemma~\ref{l42},
$$\int_0^\infty g_{R, {\bf m}}(\lambda)d\lambda +O(1/q) = 
\frac{\ell(R/{\bf m}^{[q]})}{q^{d-1}}
-(q)e_{HK}(R, {\bf m}) + O(1/q)$$
which implies 
$$\ell(R/{\bf m}^{[q]}) = 
e_{HK}(R, {\bf m})q^d + q^{d-1} \int_0^\infty g_{R, {\bf m}}(\lambda)d\lambda
+ O(q^{d-2}).$$
This gives the corollary.$\hfill\Box$

\begin{rmk}\label{mr}
Since $\partial\sP_D$ and $\partial\sP_D\cap \partial C_D$ consist of rational
$d-1$ dimensional convex polytopes, the number $\beta(R, {\bf m})$ 
(the volume of the integral) is a 
rational number and also is independent of the characteristic.
We note that the above arguement gives a direct proof 
of the result of [HMM] in our particular situation, for the graded ring $R$.
However the assertion that 
$\beta(R, {\bf m})$ is a constant and rational has been proved earlier by Bruns-Gubeladze
in [BG], for 
any  normal affine monoid.

It is an interesting problem to extend the computations here to the case of 
$R$-modules, and to determine the homomorphism $Cl(R)\longto \R$ of [HMM] in
this toric case.\end{rmk}

\section{some properties and examples}
\begin{defn}Let $R$ be a Noetherian standard graded ring of dimension $d\geq 2$ 
with the maximal homogeneous ideal  ${\bf m}$. Let 
$\ell(R_n) = \frac{e_0(R, {\bf m})}{(d-1)!}n^{d-1} + 
\tilde{e}_1(R, {\bf m})n^{d-2}+ \cdots + \tilde{e}_{d-1}(R, {\bf m})$
be the Hilbert polynomial  of $(R,{\bf m})$. 
We define  the  Hilbert  density function $F_R:[0,\infty)
\longto [0, \infty)$, of $R$  as
 $$F_R(\lambda)= \frac{e_0(R, {\bf m})}{(d-1)!}\lambda^{d-1} = 
 \lim_{n\to\infty}F_n(\lambda) 
:=\frac{1}{q^{d-1}}\ell(R_{\lfloor q\lambda\rfloor}).$$
Similarly we can define the  second Hilbert 
 density function 
 $G_R:[0,\infty)\longto \R$ as
 $$G_R(\lambda)= \tilde{e}_1(R, {\bf m})\lambda^{d-2} = 
\lim_{n\to\infty}G_n(\lambda):=\frac{1}{q^{d-2}}
\left(\ell(R_{\lfloor q\lambda\rfloor})-
 F_R\big(\frac{\lfloor q\lambda\rfloor}{q}\big)\right).$$
\end{defn}

\begin{rmk}\label{r5}
Let $R$ and $S$ be two Noetherian standard graded rings over an 
algebraically closed field $K$ of dimension $d\geq 2$ and $d'\geq 2$
with homogenous maximal ideals ${\bf m}$ and ${\bf n}$, respectively. 
Then, using the Kunneth formula for sheaf cohomology, it is easy to see
$$\tilde{e}_1(R\#S, {\bf m}\#{\bf n})=
\frac{e_0(R, {\bf m})}{(d-1)!}\tilde{e}_1(R, {\bf m})
+\frac{e_0(S, {\bf n})}{(d'-1)!}\tilde{e}_1(S, {\bf n}).$$
Hence we have
$G_{R\#S}=G_RF_S+G_SF_R.$
\end{rmk}

\begin{propose}\label{p1}
Let $(R, {\bf m})$ and $(S, {\bf n})$ be two Noetherian standard 
graded rings over an
algebraically closed field $K$ (of characteristic $p>0$) 
of dimension $d\geq 2$ and $d'\geq 2$, associated to the 
toric pairs $(X, D)$ and $(Y, D')$, resply. Then we have, 
$$G_{R\#S}-g_{R\#S, {\bf m}\#{\bf n}}=(G_R-g_{R, {\bf m}})(F_S-f_{S,{\bf n}} )
+(G_S-g_{S, {\bf n}})(F_R-f_{R, {\bf m}}).$$
Here the functions $g_{R, {\bf m}}$ and $g_{S, {\bf n}}$ denote the 
$\beta$-density functions for the pairs $(R, {\bf m})$ and $(S, {\bf n})$
respectively. The function $g_{R\#S, {\bf m}\#{\bf n}}$ denotes the 
$\beta$-density function for the pair $(R\#S, {\bf m}\#{\bf n})$.
\end{propose}

\begin{proof}
Let $\lambda\in[0, \infty)$. For $n\in N$ and $q=p^n$, we write 
$m=\lfloor q\lambda\rfloor$.
then, for the Noetherian standard graded ring $R\#S$ with homogenous
maximal ideal ${{\bf m}}\#{{\bf n}}$, we have
\begin{eqnarray*}
g_n(\lambda)=\frac{1}{q^{d+d'-3}}
\left(\ell\big(R\#S/({{\bf m}}\#{{\bf n}})^{[q]}\big)_{m}
-f_{R\#S, {\bf m}\#{\bf n}}(m/q)q^{d+d'-2}\right).
\end{eqnarray*}
By Proposition 2.17, \cite{Tri2015}, we have
\begin{eqnarray*}
g_n(\lambda)&=&\frac{1}{q^{d+d'-3}}\left[\ell(R_{m})\ell(S/{\bf n}^{[q]})_m
+\ell(S_{m})\ell(R/{\bf m}^{[q]})_m-\ell(R/{\bf m}^{[q]})_m\ell(S/{\bf n}^{[q]})_m\right]\\
&&-\left[F_R(m/q)f_{S, {\bf n}}(m/q)+F_S(m/q)f_{R, {\bf m}}(m/q)
-f_{R, {\bf m}}(m/q)f_{S, {\bf n}}(m/q)\right]q.\\
&=& \Phi_n(\lambda)+\Psi_n(\lambda)+\xi_n(\lambda),
\end{eqnarray*}
where
\begin{eqnarray*}
\Phi_n(\lambda)&=&\frac{1}{q^{d+d'-3}}\left[\ell(R_{m})
\ell(S/{\bf n}^{[q]})_m-F_R(m/q)f_{S, {\bf n}}(m/q)q^{d+d'-2}\right],\\
\end{eqnarray*}
\begin{eqnarray*}
\Psi_n(\lambda)&=&\frac{1}{q^{d+d'-3}}\left[\ell(S_{m})
\ell(R/{\bf m}^{[q]})_m-F_S(m/q)f_{R, {\bf m}}(m/q)q^{d+d'-2}\right]\\
\end{eqnarray*}
and
\begin{eqnarray*}
\xi_n(\lambda)&=&-\frac{1}{q^{d+d'-3}}\left[\ell(R/{\bf m}^{[q]})_m
\ell(S/{\bf n}^{[q]})_m-f_{R, {\bf m}}(m/q)
f_{S, {\bf n}}(m/q)q^{d+d'-2}\right].\\
\end{eqnarray*}
We have\begin{eqnarray*}
\Phi_n(\lambda)&=&\frac{1}{q^{d+d'-3}}\left[\ell(R_{m})
\ell(S/{\bf n}^{[q]})_m-F_R(m/q)f_{S, {\bf n}}(m/q)q^{d+d'-2}\right]\\
&=&\frac{1}{q^{d-1}}\ell(R_{m})\times\frac{1}{q^{d'-2}}
\left[(\ell(S/{\bf n}^{[q]})_m-f_{S, {\bf n}}(m/q)q^{d'-1})\right]\\
&&+ f_{S, {\bf n}}(m/q)\times\frac{1}{q^{d-2}}\left[\ell(R_{m})-
F_R(m/q)q^{d-1}\right].
\end{eqnarray*}
Thus 
$\lim_{n\to\infty}\Phi_n(\lambda)=F_R(\lambda)g_{S, {\bf n}}(\lambda)
+f_{S, {\bf n}}(\lambda)G_R(\lambda)$. Similarly 
$\lim_{n\to\infty}\Psi_n(\lambda)=F_S(\lambda)g_{R, {\bf m}}(\lambda)
+f_{R, {\bf m}}(\lambda)G_S(\lambda)$
and $\lim_{n\to\infty}\xi_n(\lambda)=f_{R, {\bf m}}(\lambda)g_{S, {\bf n}}(\lambda)
+ f_{S, {\bf n}}(\lambda)g_{R, {\bf m}}(\lambda)$.
\end{proof}
This implies,
$$g_{R\#S, {\bf m}\#{\bf n}}=F_Rg_{S, {\bf n}}+f_{S, {\bf n}}G_R+F_Sg_{R, {\bf m}}+f_{R, {\bf m}}G_S
+f_{R, {\bf m}}g_{S, {\bf n}}+f_{S, {\bf n}}g_{R, {\bf m}}.$$
By Remark \ref{r5}, we have
$$G_{R\#S}-g_{R\#S, {\bf m}\#{\bf n}}=(G_R-g_{R, {\bf m}})(F_S-f_{S,{\bf n}} )
+(G_S-g_{S, {\bf n}})(F_R-f_{R, {\bf m}}).$$

\vspace{5pt}

\begin{ex}\label{ex1}
Consider the toric pair $(\P^2, -K)$, where $-K$ is the 
anticanonical divisor of $\P^2$, and let $(R, {\bf m})$ be the 
associated coordinate ring with the homogeneous maximal ideal ${\bf m}$.
 The Hilbert-Kunz function of $(R, {\bf m})$
is
$\mbox{HK}(R, {\bf m})(q)=q^3+O(q).$
For $\lambda\in \R_{\geq 0}$ and for $q=p^n$, let
$f_n(R, {\bf m})$ and $f_{R, {\bf m}}$ be as in given in Theorem~\ref{hkd}
and $g_n$ and $g_{R, {\bf m}}$ be as in Definition~\ref{d5}.
 A simple 
calculation shows 
\begin{eqnarray*}
f_n(\lambda)&=&
\begin{cases}
\frac{1}{2q^2}(m+2)(m+1) & \text{if $0\leq \lambda< 1$,}\\
\frac{1}{2q^2}\left((m+2)(m+1)-3(m-q+2)(m-q+1)\right) & 
\text{if $1\leq \lambda< 2$,}\\
\frac{1}{2q^2}[(m+2)(m+1)-3(m-q+2)(m-q+1)\\
+3(m-2q+2)(m-2q+1)] & \text{if $1\leq \lambda< 2$.}\\
\end{cases}
\end{eqnarray*}
Hence
\begin{eqnarray*}
f_{R, {\bf m}}(\lambda)&=&
\begin{cases}
\frac{1}{2}\lambda^2 & \text{if $0\leq \lambda< 1$,}\\
\frac{1}{2}\lambda^2-\frac{3}{2}(\lambda-1)^2 & 
\text{if $1\leq \lambda< 2$,}\\
\frac{1}{2}\lambda^2-\frac{3}{2}(\lambda-1)^2 + 
\frac{3}{2}(\lambda-2)^2 & \text{if $1\leq \lambda< 2$.}\\
\end{cases}
\end{eqnarray*}
\begin{eqnarray*}
g_n(\lambda)&=&
\begin{cases}
\frac{1}{q}\left(\frac{3}{2}m+1\right) & \text{if $0\leq \lambda< 1$,}\\
\frac{1}{q}\left(-3m+\frac{9}{2}q-2\right) & \text{if $1\leq \lambda< 2$,}\\
\frac{1}{q}\left(\frac{3}{2}m-\frac{9}{2}q+1\right) & 
\text{if $1\leq \lambda< 2$,}\\
\end{cases}
\end{eqnarray*}
\begin{eqnarray*}
g_{R, {\bf m}}(\lambda)&=&
\begin{cases}
\frac{3}{2}\lambda & \text{if $0\leq \lambda< 1$,}\\
-3\lambda+\frac{9}{2} & \text{if $1\leq \lambda< 2$,}\\
\frac{3}{2}\lambda-\frac{9}{2} & \text{if $2\leq \lambda< 3$,}\\
\end{cases}
\end{eqnarray*}
and $\int_{0}^{\infty} g_{R, {\bf m}}(\lambda)\ d\lambda=0$\ .
\end{ex}

\begin{ex}\label{ex2}
We compute the $\beta$-density function for the Hirzebruch surface $X=\F_a$ 
with parameter $a\in{\mathbb{N}}$, which is a ruled surface over 
$\mathbb{P}^1_K$, where $K$ is a field of characteristic $p> 0$. 
See [?] for a detailed description of the surface as a 
toric variety. The $T$-Cartier divisors are given by 
$D_i=V(v_i), i=1,2 ,3 , 4$, where $v_1=e_1, v_2=e_2, v_3=-e_1+ae_2, 
v_4=-e_2$ and $V(v_i)$ denotes the $T$-orbit closure corresponding to 
the cone generated by $v_i$. We know the Picard group is generated 
by $\{D_i\ : i=1, 2,3, 4\}$ over $\mathbb{Z}$. One can check the 
only relations in $\text{Pic}(X)$ can be described by 
$D_3\thicksim D_1$ and $D_2\thicksim D_4-aD_1$. Therefore 
$\text{Pic}(X)=\mathbb{Z}D_1\oplus\mathbb{Z}D_4$. One can use 
standard method of toric geometry to see that $D=cD_1+dD_4$ is 
ample if and only if $a, c>0$. Then 
$P_D=\{(x,y)\in M_\mathbb{R}\ |\ x\geq -c, y\leq d, x\leq ay\}$ 
and $\alpha^2=\text{\mbox{Vol}}(P_D)=cd+\frac{ad^2}{2}$.  For a detailed 
analysis of Hilbert-Kunz function and Hilbert-Kunz density 
function of $\F_a$, see [T1]. In the following we use results  
from [T1] to calculate $\{f_n(R, {\bf m})\}$ and $f_{R, {\bf m}}$ in order to 
describe the $\beta$-density function. One can also use the computations from 
Example 7.2 of [MT].

If $c \geq d$ then $f_{R, {\bf m}}(q) = (*)$, where
\begin{eqnarray*}
(*)&=&q^3 (c+\frac{ad}{2})\left[ \frac{d}{3}+\frac{(d+1)d}{6c(ad+c)}
+\frac{1}{2}+\frac{1}{6d}\right]\\
&& + q^2\left( (c+\frac{ad}{2})(d+1)\left[ \frac{1}{4c} 
+ \frac{1}{4(ad+c)} -\frac{d}{2c(ad+c)} -\frac{1}{ 2d}\right] 
+\frac{d+ 1}{2}\right) +O(q)
\end{eqnarray*}
and if $c<d$ then
$f_{R, {\bf m}}(q)=$\begin{eqnarray*}
&&(*)+q^3d\left(c+\frac{(d+1)a}{2}\right)\left(\frac{(a+1)^3}{6a(ad + c)} 
-\frac{1}{6ac} -\frac{a}{ 6d} -\frac{1}{ 2d} +\frac{c}{6d^2}\right)\\
&&-q^2d\left(c+\frac{(d+1)a}{2}\right)
\left[ \frac{(a+2)(a+1)^3}{4a(c + ad)} +\frac{a-2}{ 4ac} 
-\frac{a+2}{4d} -\frac{1}{ d} +\frac{c}{ 2d^2}\right] +O(q).
\end{eqnarray*}

An easy but tedious calculation shows that, for $c\geq d$
\begin{eqnarray*}
                                   g_{R, {\bf m}}(\lambda)&=&    
\begin{cases} 
\left(c+\frac{ad}{2} +d\right)\lambda & \text{ if $ 0\leq \lambda < 1$} \\\\
-\left(c+\frac{ad}{2}+d\right)\left(cd+\frac{ad^2}{2}+c+\frac{ad}{2}+d\right)
\lambda\\
+ \left(c+\frac{ad}{2} +d\right)(d+1)\left(c+\frac{ad}{2}+1\right)  
&  \text{ if $ 1\leq \lambda < 1+\frac{1}{c+ad}$}\\\\
d(d+1)-d\left(c+\frac{ad}{2} +d\right)\lambda\\
+\left(c+\frac{ad}{2}\right)(d+1)
\left(\frac{1}{2}-\frac{1}{a}\right)(c+1-c\lambda) 
&  \text{ if $ 1+\frac{1}{c+ad}\leq \lambda < 1 +\frac{1}{c}$}\\\\
d(d+1)-d\left(c+\frac{ad}{2} +d\right)\lambda & 
\text{ if $1+\frac{1}{c}\leq \lambda < 1 +\frac{1}{d}$ ,}\\
 \end{cases}
\end{eqnarray*}

and for $c \leq d $ 
\begin{eqnarray*}
g_{R, {\bf m}}(\lambda)&=&    
\begin{cases} 
\left(c+\frac{ad}{2} +d\right)\lambda & \text{ if $ 0\leq \lambda < 1$} \\\\
-\left(c+\frac{ad}{2} +d\right)
\left(cd+\frac{ad^2}{2}+c+\frac{ad}{2}+d\right)\lambda \\
+ \left(c+\frac{ad}{2} +d\right)(d+1)\left(c+\frac{ad}{2}+1\right)  
&  \text{ if $ 1\leq \lambda < 1+\frac{1}{c+ad}$}\\\\
d(d+1)-d\left(c+\frac{ad}{2} +d\right)\lambda\\
+\left(c+\frac{ad}{2}\right)(d+1)
\left(\frac{1}{2}-\frac{1}{a}\right)(c+1-c\lambda) 
&  \text{ if $ 1+\frac{1}{c+ad}\leq \lambda < 1 +\frac{1}{d}$}\\\\
\left(c+\frac{ad}{2}\right)(d+1)\left(\frac{1}{2}-\frac{1}{a}\right)
(c+1-c\lambda)\\
+d\left(c+\frac{ad}{2}+\frac{a}{2}\right)
\left(c+\frac{ad}{2}+d\right)(\lambda-1)\\
-d\left(2+\frac{a}{2}\right)\left(c+\frac{ad}{2}+\frac{a}{2}\right) 
& \text{ if $1+ \frac{1}{d}\leq \lambda < 1 +\frac{a+1}{ad+c}$}\\\\
c\left(\frac{1}{2}-\frac{1}{a}\right)(c+1-c\lambda)
&  \text{ if $1 +\frac{a+1}{ad+c}\leq \lambda < \frac{1}{c}$ .}
\end{cases}
\end{eqnarray*}
\end{ex}

\section{Appendix}
We recall  the following notion of
relative volume for a convex polytope, given by
R. Stanley (see page 238 in [St])).

\begin{defn}\label{rv} Let $Q\subset \R^d$ be an integral convex polytope of
dimension $m$. Let $A(Q)$ be the affine span of $Q$. Then $A(Q)\cap \Z^d$ is
an abelian group of rank $m$.
Choose an invertible affine  transformation $\varphi:A(Q)\longto \R^m$
such that
$\varphi(A(Q)\cap \Z^d) = \Z^m$. Then the image $\varphi(Q)$ of $Q$ is an integral
polytope and relative volume of $Q$ is defined to be volume of $\varphi(Q)$.

If $Q$ is an $m$-dimensional rational polytope then there is  $n>0$ such that
 $nQ$ is an
integral polytope. We define $\mbox{rVol}_{m}(Q) := \mbox{rVol}_{m}(nQ)/{n^{m}}$
and $\mbox{rVol}_{m_1}(Q) := 0$ if $m_1 > m$.

Moreover,  if $Q = \cup_iQ_i$, is a finite union of
 $d'$-dimensional convex rational polytopes $Q_i$ such that
$\dim{(Q_i\cap Q_j)}< d'$, for $Q_i\neq Q_j$, then
we define  $\mbox{rVol}_{d'}Q = \sum_i\mbox{rVol}_{d'}Q_i$ and
$\mbox{rVol}_{d''}Q = 0$, if $d''> d'$.
One can show that $\mbox{rVol}_{d'}Q$ is independent 
of the choice of the finite decomposition $Q = \cup_iQ_i$. 

\end{defn}

\begin{lemma}\label{rvol}Let $F$ be a rational convex polytope of dimension
$d-1$ in $\R^d$
satisfying the following:
$A(F)\cap \{z=i\}\cap \Z^d \neq \phi$, for every $i\in \Z$,
where
$A(F)$ denotes the affine hull of the polytope  $F$, and $z$ is the coordinate function on 
$\R^d$. Let $\varphi:A(F)\longto
\R^{d-1}$ be an invertible affine  transformation such that
$\varphi(A(F)\cap \Z^d) = \Z^{d-1}$.

Then there is $z_1 \in A(F)\cap \{z=1\}\cap \Z^d$ and
$\{x_1, \ldots, x_{d-2}\}\in A(F)\cap \{z=0\}\cap \Z^d$ such
that $\{\varphi(x_1),\ldots, \varphi(x_{d-2}), \varphi(z_1)\}$ is a basis
of $\R^{d-1}$ and
$$\textnormal{\mbox{rVol}}_{d-2}~(F\cap \{z=i\}) =
\textnormal{\mbox{Vol}}_{d-2}~(\varphi(F)\cap \{\pi_\varphi(\varphi(z_1))=i\}),$$
where $\pi_\varphi:\R^{d-1} = \varphi(A(F)) \longto \R$ is the map given by
$\sum\lambda_i\varphi(x_i)+\lambda_{d-1}\varphi(z_1)\to \lambda_{d-1}$.
\end{lemma}
\begin{proof}\underline{Reduction}:
By hypothesis, we can choose
 $x_0\in A(F)\cap \Z^d$. Let $y_0 = \varphi(x_0)$.
Let $\psi_{x_0}:\R^d\longto \R^d$ and $\psi'_{-y_0}:\R^{d-1}\longto \R^{d-1}$
 denote the translation maps given by the elements $x_0$ and $-y_0$
respectively.
Then replacing $\varphi$ by $\psi'_{-y_0}\circ \varphi \circ\psi_{x_0}:A(F)-x_0\longto
\R^{d-1}$ and $A(F)$ by $A(F)-x_0$., we can assume that $\varphi$ is a linear
transformation and $A(F)$ is an $\R$-vector space.
Let $V_n = A(F)\cap \{z=n\}$ then $V_0$ is a $d-1$ dimensional
vector subspace. We choose $z_1\in V_1$.

\vspace{5pt}
\noindent{\bf Claim} (1) $V_n = V_0+nz_1$ and

(2) $\varphi(V_n) = \varphi(A(F))\cap \{\pi_\varphi\varphi(z_1) = n\}$.
There exists a basis $\{\varphi(x_1), \ldots, \varphi(x_{d-2}), \varphi(z_1)\}
\in \Z^{d-1}$ of $\varphi(A(F))$.

\noindent{Proof of the claim}: (1)\quad Let $x_1, \ldots, x_{d-2}$ be a set of
generators of the free abelian  group $V_0\cap \Z^d$.
Then $\{x_1, \ldots, x_{d-2}\}$ is a vector space basis for $V_0$.
Therefore
$\{\varphi(x_1), \ldots, \varphi(x_{d-2}), \varphi(z_1)\}$ is a basis for
$\varphi(A(F)) = \R^{d-1}$. Hence $\varphi(A(F)) = \sum_i\R\varphi(x_i)+\R\varphi(z_1)$.
It is obvious that $V_0+nz_1 \subseteq V_n$. Let $y\in V_n$ then
$y=\sum_i\lambda_ix_i+\lambda_0 z_1$. Comparing the $(d-1)^{th}$ coordinate we get
$\lambda_0 = n$ and therefore $y\in V_0+nz_1$.
(b)\quad Let $y\in V_n = V_0+nz_1$. Then $y= (\sum_i\lambda_ix_i) + nz_1$.
This implies $\varphi(V_n) = \sum_i\R\varphi(x_i)+\R\varphi(z_1) \subseteq
\varphi(A(F))\cap
\{\pi_\varphi(\varphi(z_1)) = n\}$.  This proves the claim.

Since
$\varphi(V_n) =  (\sum_i\R\varphi(x_i))+n\varphi(z_1)$
and $V_n\cap\Z^d = \sum_{i=1}^{d-2}\Z{x}_i+\varphi({z}_1)$,
we have $\varphi(V_n\cap \Z^d)\subseteq \varphi(V_n)\cap \Z^{d-1}$.
Now
$$\varphi(A(F)\cap \Z^d)= \cup_n\varphi(V_n\cap\Z^d) \subseteq
\cup_n(\varphi(V_n)\cap\Z^{d-1})\subseteq \varphi(A(F)\cap\Z^{d-1}),$$
This implies  $\varphi(V_n\cap \Z^d) = \varphi(V_n)\cap \Z^{d-1}$, for all
$n\in \Z$. Hence  $\mbox{rVol}_{d-2}(V_n) = \mbox{Vol}_{d-2}\varphi(V_n) =
\mbox{Vol}_{d-2}[\varphi(A(F))\cap \{\pi_\varphi(\varphi(z_1)) =n\}]$.
This proves the lemma.
\end{proof}

\begin{lemma}\label{la}If $Q$ is a convex rational polytope in $\R^d$ and 
$Q_\lambda = Q\cap \{z = \lambda\}$, for $\lambda\in \R$ then there is a
constant $C_Q$ (independent of $\lambda$)  such that 
$i(Q_\lambda, n) = c_{\lambda}(n)n^{dim~Q_\lambda}$ with $|c_{\lambda}(n)|\leq C_Q$.
\end{lemma}
\begin{proof}For $\lambda \in \R$, let $d_\lambda = \dim~Q_\lambda$. 
Let $P_\lambda$ be an integral 
$d_\lambda$-dimensional cube with length of each side $=c_\lambda$ and 
$Q_\lambda \subseteq P_\lambda$.  Then  
$i(Q_\lambda, n) \leq i(P_\lambda, n) \leq  c_\lambda^{d_\lambda}n^{d_\lambda}$.
Since $Q$ is a bounded set, there exists a constant $C'_Q$ such that, 
for any $\lambda$ we can choose $P_\lambda$ with 
$|c_\lambda|\leq C'_Q$. Hence the lemma follows by taking 
$C_Q = (C'_Q)^{\dim~Q}$.\end{proof}


\begin{thebibliography}{566}

\bibitem[BG]{BG} W.Bruns, J. Gubeladze, {\it Divisorial linear 
algebra of normal semigroup rings}. Algebr. Represent. Theory 6, 139-168, 2003. 
\bibitem[CK]{CK} C.Y. Chan, K. Kurano, Kazuhiko, {\it Hilbert-Kunz 
functions over rings regular in codimension one}. 
Comm. Algebra 44 (2016), no. 1, 141-163. 

\bibitem[CLS]{CLS2011} D. A. Cox, J. B. Little, H. K. Schenck,
\textit{Toric varieties}, Graduate Studies in Mathematics,
American Mathematical Society, Providence, RI, (2011).

\bibitem[Eh]{Eh} E. Ehrhart, \textit{Sur les poly$
\grave{e}$dres rationnels homoth$\acute{e}$tiques $\grave{a}$
n dimensions}, C. R. Acad. Sci. Paris Ser. A 254 (1962) 616-618.

\bibitem[Et]{Et} K. Eto, \textit{Multiplicity and Hilbert-Kunz 
multiplicity of monoid rings}, Tokyo J. Math, 25 (2002), no. 2, 241-245.


\bibitem[F]{Ful1993} W. Fullton, \textit{Introduction to toric
varieties}, Annals of Mathematics Studies, 131, The William H. Roever
Lectures in Geometry. Princeton University Press, Princeton, NJ, (1993).


\bibitem[HL]{HenkLinke2015} M. Henk, E. Linke, \textit{Note on the
coefficients of rational Ehrhart quasi-polynomials of Minkowski-sums}.
Online J. Anal. Comb. No. 10 (2015), 12 pp. 52B20 (05A15 11P21 52A39)

\bibitem[HY]{HY} M.Hochster, Y.Yao, {\it Second coefficients of Hilbert-Kunz 
functions for domains}. 
Preliminary preprint available at http://www.math.lsa.umich.edu/ ̃hochster/hk.pdf. 


\bibitem[HMM]{HMM2004}C. Huneke, M. A. McDermott, P. Monsky,
\textit{Hilbert-Kunz functions for normal rings}. Math. Res. Lett.,
11: 539-546, 2004.

\bibitem[Ku]{Kurano2006} K. Kurano, The singular Riemann-Roch theorem
and Hilbert-Kunz functions, J. Algebra 304
(2006), 487-499.

\bibitem[L]{Linke2011}E. Linke. \textit{Rational Ehrhart quasi-polynomials}.
Journal of Combinatorial Theory, Series A, 118(7):1966-1978, 2011.

\bibitem[M]{McM78} P. McMullen, \textit{Lattice invariant valuations
on rational polytopes}, Arch. Math. 31 (1978/1979) 509-516.


\bibitem[MT]{MT}M. Mondal, V. Trivedi, \textit{Hilbert-Kunz Density
function and and Asymptotic Hilbert-Kunz multiplicity for Projective
Toric Varieties}, arXiv:1707.05959v3.

\bibitem[Mo]{M} P. Monsky, \textit{The Hilbert-Kunz function.},
Math. Ann. 263 (1983), no. 1, 43-49.



\bibitem[Sc]{Sch2013} R. Schneider. Convex bodies: The Brunn-Minkowski
Theory. 2nd expanded edition, Cambridge Uni- versity Press, 2013.

\bibitem[S]{S} R. P. Stanley, \textit{Enumerative Combinatorics vol. 1},
Cambridge University press (1997).

\bibitem[T1]{Tri2016}V. Trivedi, \textit{Hilbert-Kunz functions of a
Hirzebruch surface}, Journal of Algebra 457 (2016) 405-430.

\bibitem[T2]{Tri2015}V.\hspace{.12cm}Trivedi, \textit{Hilbert-Kunz
Density Function and Hilbert-Kunz Multiplicity}, arXiv:1511.02941,
To appear in Transactions AMS.

\bibitem[T3]{T3}V. Trivedi, {\it Asymptotic Hilbert-Kunz multiplicity}, 
J. Algebra 492 (2017), 498-523.

\bibitem[T4]{T4}V. Trivedi, {\it Towards  Hilbert-Kunz multiplicity 
in characteristic $0$}, arXiv:1601.01775.


\end{thebibliography}
\end{document}